\documentclass[12pt,a4paper]{amsart}
\usepackage[latin1]{inputenc}
\usepackage{amssymb, amsmath}
\usepackage{geometry}
\usepackage{enumerate}
\newtheorem{theorem}{Theorem}[section]
\newtheorem{definition}[theorem]{Definition}
\newtheorem{lemma}[theorem]{Lemma}
\newtheorem{corollary}[theorem]{Corollary}

\newtheorem{remark}[theorem]{Remark}

\begin{document}
\title[Mixed norm spaces and rearrangement invariant estimates]{Mixed norm spaces and rearrangement invariant estimates}
\author{Nadia Clavero}
\address{Department of  Applied Mathematics and Analysis, University of Barcelona, Gran Via 585, E-08007 Barcelona, Spain}
\email{nadiaclavero@ub.edu}

\author{Javier Soria}
\address{Department of  Applied Mathematics and Analysis, University of Barcelona, Gran Via 585, E-08007 Barcelona, Spain}
\email{soria@ub.edu}


\thanks{Both authors have been partially supported by the Spanish Government Grant MTM2010-14946.}

\begin{abstract} Our main goal in this work  is to  further improve the mixed norm estimates due to Fournier~\cite{Fournier}, and also   Algervik and Kolyada~\cite{Robert-Viktor}, to more general rearrangement invariant (r.i.) spaces. In particular we find the optimal  domains and the optimal ranges for these embeddings between mixed norm spaces and r.i. spaces.
\end{abstract}
\maketitle

\thispagestyle{empty}

\section{Introduction}

Estimates on mixed norm spaces  $\mathcal{R}(X,Y)$ (see Definition~\ref{definition: mixed norm spaces})  already appeared in  the works of Gagliardo \cite{Gagliardo} and Nirenberg \cite{Nirenberg},  to prove an endpoint case of the classical Sobolev embeddings. However, a more systematic approach to these spaces was first shown explicitly by Fournier \cite{Fournier}.

Recall that  the Sobolev space $W^{1}L^{p}(I^{n}),$ $1\leq p<\infty,$  consists of all functions in $L^{p}(I^{n})$ whose first-order distributional derivatives also belong to $L^{p}(I^{n}).$  We write  $W^{1}_{0}L^{p}(I^{n})$ for the closure of  smooth function, with compact support, in $W^{1}L^{p}(I^{n}).$

The classical Sobolev embedding theorem claims that
\begin{align}\label{eq: inclusion sobolev}
W^{1}_{0}L^{p}(I^{n})\hookrightarrow L^{pn/(n-p)}(I^{n}),\ \ 1\leq p<n.
\end{align}
 Sobolev \cite{Sobolev} proved this embedding for $p>1,$ but his method, based on integral representations, did not work when $p=1.$  That case was settled affirmatively by Gagliardo \cite{Gagliardo} and  Nirenberg~\cite{Nirenberg}, who first  observed that
\begin{align}\label{eq: Gagliardo-Niremberg}
W^{1}_{0}L^{1}(I^{n})\hookrightarrow \mathcal{R}(L^{1},L^{\infty}),
\end{align}
and then, using an iterated form of H\"older's inequality,  completed the proof; i.e.,
 \begin{align*}
W^{1}_{0}L^{1}(I^{n})\hookrightarrow \mathcal{R}(L^{1},L^{\infty})\hookrightarrow L^{n^{\prime}}(I^{n}).
\end{align*}

Later, a new approach based on properties of mixed norm spaces was introduced by Fournier \cite{Fournier} and was subsequently developed, via different
methods, by various authors, including Blei \cite{Blei-Fournier}, Milman \cite{Milman-kfunctional-mixed-norm}, Algervik and Kolyada \cite{Robert-Viktor} and Kolyada \cite{Viktor2012, Viktor2013}. To be more precise, the central part of Fournier's work was  to prove
\begin{align}\label{eq: Fournier}
\mathcal{R}(L^{1},L^{\infty})\hookrightarrow L^{n^{\prime},1}(I^{n}),
\end{align}
and then taking into account \eqref{eq: Gagliardo-Niremberg}, he obtained the following improvement of \eqref{eq: inclusion sobolev}:
 \begin{align}\label{eq: inclusion sobolev Lorentz}
W^{1}_{0}L^{1}(I^{n})\hookrightarrow L^{n^{\prime},1}(I^{n}).
\end{align}
The embedding \eqref{eq: inclusion sobolev Lorentz} is due to  Poornima \cite{Poornima}, and it can be also traced in the work of Peetre \cite{Peetre} (the case $W_0^1L^p$, $p>1$), and Kerman and Pick \cite{Kerman-Pick}, where a characterization of Sobolev  embeddings  in rearrangement invariant  (r.i.) spaces  was obtained.

Since the embedding \eqref{eq: Fournier} was first proved, many other proofs and different extensions have appeared in the literature. In particular, relations between mixed norm  spaces of Lorentz spaces were studied in \cite{Robert-Viktor}, where it was shown, for instance, 
 \begin{align}\label{eq: introduction viktor-robert}
 \mathcal{R}(L^{1},L^{\infty})\hookrightarrow \mathcal{R}(L^{(n-1)^{\prime},1},L^{1}),\ \ n\geq 2.
 \end{align}

All these works provide us a strong motivation to better understand the embeddings between mixed norm spaces, as well as to provide a characterization of   \eqref{eq: Fournier} for more general r.i. spaces.
\medskip

The paper is organized as follows. A precise definition of r.i. spaces, and also other definitions and properties concerning function spaces to be used throughout the paper can be found in Section~\ref{sec: preliminaries}.

In Section~\ref{mixed norm spaces} we introduce  the  Benedek-Panzone spaces and the mixed norm spaces. Here, among other things, we  find an explicit formula for the Peetre $K$-functional for the couple of mixed norm spaces $(\mathcal{R}(X,L^{\infty}),L^{\infty}).$ 

In Section~\ref{section: Embeddings between mixed norm spaces}, we explore connections between  mixed norm spaces. It is important to observe that, for any $k\in \left\{1,\ldots,n\right\},$ 
\begin{align}\label{eq: equivalencia}
\mathcal{R}_{k}(X_{1},Y_{1})\hookrightarrow \mathcal{R}_{k}(X_{2},Y_{2})\Leftrightarrow \begin{cases}X_{1}(I^{n-1})\hookrightarrow X_{2}(I^{n-1}),\\ Y_{1}(I)\hookrightarrow Y_{2}(I),\end{cases}
\end{align}
(see Lemma~\ref{lema: inclusion the Benedek-Panzone spaces}). That is, looking at each specific component, the  embeddings are trivial. However,  there are examples, for instance \eqref{eq: introduction viktor-robert}, showing that if in \eqref{eq: equivalencia} we replace Benedek-Panzone spaces by  the global mixed norm spaces, then the corresponding equivalence is not longer true. This fact illustrates that mixed norm spaces may have a much more complicated structure than Benedek-Panzone spaces. Therefore, it is natural to analyze  embeddings between mixed norm spaces. 

Motivated by this problem  we  find necessary and sufficient conditions for the existence of the following types of embeddings:
\begin{align*}
\begin{array}{lll}
\mathcal{R}(L^{\infty},Y_{1})\hookrightarrow \mathcal{R}(X_{2},Y_{2}),&  \mathcal{R}(X_{1},L^{\infty})\hookrightarrow \mathcal{R}(X_{2},L^{\infty}),&  \mathcal{R}(X_{1},Y)\hookrightarrow \mathcal{R}(X_{2},Y).
\end{array}
\end{align*}
After this  discussion our analysis focuses on a particular embedding:
\begin{align}\label{eq: introduction R(X1,Loo)-> R(X2,Lp1)}
 \mathcal{R}(X_{1},L^{\infty})\hookrightarrow \mathcal{R}(X_{2},L^{1}).
\end{align}
To be more specific, Theorem~\ref{teo: rango R(X1,Loo)->R(X2,L1)} provides a characterization of the smallest mixed norm space of the form $\mathcal{R}(X_2,L^{1})$ in \eqref{eq: introduction R(X1,Loo)-> R(X2,Lp1)}, once the mixed norm space $\mathcal{R}(X_1,L^{\infty})$ is given. To finish, we make some comments for the case when $L^1$ is replaced by $L^{p,1}$ (see also \cite{Robert-Viktor}).

Section~\ref{section: Fournier embeddings} is devoted to the study  of the embedding \eqref{eq: Fournier} for more general r.i. spaces. In particular, Theorem~ \ref{theorem: mixta inclusion R(X,Loo)->Z}  gives necessary and sufficient conditions for the following embedding to hold:
\begin{align}\label{eq: intro R(X,Loo)->Z}
\mathcal{R}(X,L^{\infty})\hookrightarrow Z(I^{n}).
\end{align}
A general consequence of Theorem~\ref{theorem: mixta inclusion R(X,Loo)->Z} is contained in Theorem~\ref{theorem: dominio R(X,Loo)->Z}, which provides
a characterization of the largest space of the form $\mathcal{R}(X,L^{\infty})$ in \eqref{eq: intro R(X,Loo)->Z}, once the r.i. space $Z(I^{n})$ is given.  Finally, for a fixed mixed norm space $\mathcal{R}(X,L^{\infty})$, Theorem~\ref{teo: rango R(X,Loo)->Z} describes the smallest r.i. space for which \eqref{eq: intro R(X,Loo)->Z} holds.

Some remarks about the notation: The measure of the unit ball in $\mathbb R^{n}$ will be represented by $\omega_{n}.$ As usual, we  use the symbol $A\lesssim B$ to indicate that there exists a universal positive constant $C$, independent of all
important parameters, such that $A\leq CB$. The equivalence $A\approx B$ means that $A\lesssim B$ and $B\gtrsim A$. Finally, the arrow  $\hookrightarrow$ stands for a continuous embedding.

\section{Preliminaries}\label{sec: preliminaries}
We collect in this section some basic notations and concepts that will be useful in what follows.

Let $n\in\mathbb N,$ with $n\geq 1$ and let $I\subset \mathbb R$ be a finite interval. We write  $\mathcal{M}(I^{n})$ for the set of all real-valued measurable functions on $I^{n}.$ 

Given $f\in\mathcal{M}(I^{n}),$ its distribution function $\lambda_{f}$  is defined by
\begin{align*}
\lambda_{f} (t) = |\bigl\{x\in I^{n}: |f (x)| > t\bigr\}|,\ \ t\geq  0,
\end{align*}
(where $|\cdot|$ denotes the Lebesgue measure) and the decreasing rearrangement $f^{*}$  of $f$ is given by
\begin{align*}
f^{*}(t) =  \inf\bigl\{s\geq 0: \lambda_{f}(s)\leq t\bigr\},\ \ t\geq 0.
\end{align*}
It is easily seen that if $g$ is any radial function on $\mathbb R^{n}$ of the form $g(x)=f^{*}(\omega_{n}|x|^{n}),$ for some $f\in\mathcal{M}(I^{n}),$ then $g^{*}=f^{*}.$

As usual, we shall use the notation $f^{**}(t)=t^{-1}\int^{t}_{0}f^{*}(s)ds.$

A basic property of rearrangements is the Hardy-Littlewood inequality (cf. e.g. \cite[Theorem II.2.2]{Bennett}), which says: 
\begin{align*}
\int_{I^{n}}|f(x)g(x)|dx\leq \int^{|I|^{n}}_{0}f^{*}(t)g^{*}(t)dt,\ \ f,g\in\mathcal{M}(I^{n}).
\end{align*}

A rearrangement invariant Banach function space $X(I^{n})$  (briefly an r.i. space) is the collection of all $f\in\mathcal{M}(I^{n})$ for which $\bigl\|f\bigr\|_{X(I^{n})}<\infty,$ where $\bigl\|\cdot\bigr\|_{X(I^{n})}$ satisfies the following properties:
 \begin{enumerate}[ ({A}1)]
	\item $\bigl\|\cdot\bigr\|_{X(I^{n})}$ is a norm;
	\item if $0\leq g\leq f\ \textnormal{a.e.,}$ then  $ \bigl\|g\bigr\|_{X(I^{n})}\leq \bigl\|f\bigr\|_{X(I^{n})};$
	\item if $0\leq f_{j}\uparrow f\ \textnormal{a.e.,}$ then $\bigl\|f_{j}\bigr\|_{X(I^{n})}\uparrow \bigl\|f\bigr\|_{X(I^{n})};$
	\item $\bigl\|\chi_{I^{n}}\bigr\|_{X(I^{n})}<\infty;$
	\item $\int_{I^{n}}|f(x)|dx\lesssim \bigl\|f\bigr\|_{X(I^{n})};$ 
\item if $f^{*}=g^{*},$ then $\bigl\|f\bigr\|_{X(I^{n})}=\bigl\|g\bigr\|_{X(I^{n})}.$
\end{enumerate}

Given an r.i. space $X(I^{n}),$ the set 
\begin{align*}
X^{\prime}(I^{n})=\Bigl\{f\in\mathcal{M}(I^{n}):\int_{I^{n}}|f(x)g(x)|dx<\infty,\ \textnormal{for any $g\in X(I^{n})$}\Bigr\},
\end{align*}
equipped with the norm
\begin{align*}
\bigl\|f\bigr\|_{X^{\prime}(I^{n})}=\sup_{\bigl\|g\bigr\|_{X(I^{n})}\leq 1}\int_{I^{n}}|f(x)g(x)|dx
\end{align*}
is called the associate space of $X(I^{n})$. It turns out that $X^{\prime}(I^{n})$ is again an r.i. space \cite[Theorem~I.2.2]{Bennett}. The fundamental function of an r.i. space $X(I^{n})$ is given by
\begin{align*}
\varphi_{X}(t)=\bigl\|\chi_{E}\bigr\|_{X(I^{n})},
\end{align*}
where $\left|E\right|=t$ and $\chi_{E}$ denotes the characteristic function of the set $E\subset I^{n}$. It is known \cite[Theorem~5.2]{Lenka} that if $X(I^{n})$ is an r.i. space, then
\begin{align}\label{eq: abs. cont} 
X(I^{n})\neq L^{\infty}(I^{n})\Leftrightarrow \lim_{t\rightarrow0^{+}}\varphi_{X}(t)=0.
\end{align}
The Lorentz space $\Lambda_{\varphi_{X}}$ consists of all $f\in\mathcal{M}(I^{n})$ for which the expression
\begin{align*}
\left\|f\right\|_{\Lambda_{\varphi_{X}}}=\left\|f\right\|_{L^{\infty}(I^{n})}\varphi_{X}(0^{+})+\int^{|I|^{n}}_{0}f^{*}(t)\varphi^{\prime}_{X}(t)dt
\end{align*}
is finite. It is well-known \cite[Theorem II.5.13]{Bennett} that if $X(I^{n})$ is an r.i. space then,
\begin{align*}
\Lambda_{\varphi_{X}}\hookrightarrow X(I^{n}).
\end{align*}
A basic tool for working with r.i. spaces is the Hardy-Littlewood-P\'olya Principle  (cf. \cite[Proposition~II.3.6]{Bennett}) which asserts that if $f\in X(I^{n})$ and 
\begin{align*}
\int^{t}_{0}g^{*}(s)ds\leq \int^{t}_{0}f^{*}(s)ds,\ \ 0<t<|I|^{n},
\end{align*}
then $g\in X(I^{n})$ and $\bigl\|g\bigr\|_{X(I^{n})}\leq \bigl\|f\bigr\|_{X(I^{n})}.$

For later purposes, let us recall the Luxemburg \textit{Representation Theorem} \cite[Theorem~II.4.10]{Bennett}. It says that given an r.i. space $X(I^n)$, there exists another r.i. space $\overline{X}(0,|I|^{n})$  such that 
$$f\in X(I^{n})\Longleftrightarrow f^{*}\in\overline{X}(0,|I|^{n}),$$
 and in this case $\bigl\|f\bigr\|_{X(I^{n})}=\bigl\|f^{*}\bigr\|_{\overline{X}(0,|I|^{n})}.$

Next, we define the Boyd indices of an r.i. space. First we introduce
the dilation operator:
\begin{align*}
E_{t}f(s)=\begin{cases}
f(s/t),& \textnormal{if $0\leq s\leq \min (|I|^{n},t|I|^{n}),$}\\
0,&\textnormal{otherwise,}
\end{cases}\ \ t>0,\ f\in\mathcal{M}(0,|I|^{n}).
\end{align*}
Let us recall that the operator $E_{t}$ is bounded on $\overline{X}(0, |I|^{n}),$ for every r.i. space $X(I^{n})$ and for every $t>0.$ 

By means of the norm
of $E_{t}$ on $\overline{X}(0,|I|^{n})$, denoted by $h_{X}(t),$ we define the lower and upper Boyd indices of $X(I^{n})$  as
\begin{align*}
\underline{\alpha}_{X}=\sup_{0<t<1}\frac{\log h_{X}(t)}{\log(t)}\ \ \textnormal{and}\ \  \overline{\alpha}_{X}=\inf_{1<t<\infty}\frac{\log h_{X}(t)}{\log(t)}.
\end{align*}
It is easy to see that $0\leq \underline{\alpha}_{X}\leq \overline{\alpha}_{X}\leq 1.$

Basic examples of r.i. spaces are the Lebesgue spaces $L^{p}(I^{n}),$ $1\leq p\leq \infty.$ Another important class of r.i. spaces is provided by the Lorentz  spaces. We recall that the Lorentz space $L^{p,q}(I^{n})$, with $1<p<\infty$ and $1\leq q\leq \infty$ or $p=q=\infty,$ is
the r.i. space consisting of all $f\in\mathcal{M}(I^{n})$ for which the quantity
\begin{align*}
\bigl\|f\bigr\|_{L^{p,q}(I^{n})}=
\begin{cases}\displaystyle{
\bigg(\int^{|I|^{n}}_{0}\bigl[s^{1/p}f^{*}(s)\bigl]^{q}\frac{ds}{s}\bigg)^{1/q}},& \textnormal{if $q<\infty,$}\\
\displaystyle{\sup_{0<t<|I|^{n}}t^{1/p}f^{*}(t)},&\textnormal{if $q=\infty$}
\end{cases}
\end{align*}
is finite. Observe that $L^{p,p}(I^{n})=L^{p}(I^{n}).$ For a comprehensive treatment of r.i. spaces we refer the reader to \cite{Bennett,Grafakos}.

Finally,  let us  recall some special results from Interpolation Theory, which we shall need in what follows.  For further information on this topic see \cite{Bennett, Bergh,Brudny}.

Given a pair of compatible Banach spaces $(X_{0}, X_{1})$ (compatible in the sense that they are continuously embedded into a common Hausdorff topological vector space), their $K$-functional is defined, for each $f\in X_{0}+X_{1},$ by
\begin{align*}
K(f,t; X_{0},X_{1}):=\inf_{f=f_{0}+f_{1}}(\bigl\|f_{0}\bigr\|_{X_{0}}+t\bigl\|f_{1}\bigr\|_{X_{1}}),\ \ t>0.
\end{align*}

The fundamental result  concerning the $K$-functional is:
\begin{theorem}\label{interpolkf} Let $(X_{0},X_{1})$ and $(Y_{0},Y_{1})$ be two compatible pairs of Banach spaces and let $T$ be a sublinear operator satisfying
\begin{align*}
T\:: X_{0}\rightarrow Y_{0},\ \  \textnormal{and}\ \  T\:: X_{1}\rightarrow Y_{1}.
\end{align*}
Then, there exists a constant $C>0$ (depending only on the norms of $T$ between $X_{0}$ and $Y_{0}$ and between $X_{1}$ and $Y_{1}$) such that
\begin{align*}
K(Tf,t; Y_{0},Y_{1})\leq C K(f, C t; X_{0}, X_{1}),\ \ \textnormal{for every $f\in X_{0}+X_{1}$ and $t>0.$}
\end{align*} 
\end{theorem}

%The $K$-functional for pairs of Lorentz spaces $L^{p,q}(I^{n})$ is given, up to equivalence, by the following result.
%
%\begin{theorem}(Holmstedt's formulas \cite[Theorem 4.2]{Holmstedt})\label{theorem: Holmstedt formulas} Let  $p_{0}=q_{0}=1$ or $1<p_{0}<\infty$ and $1\leq q_{0}<\infty.$ Let $1/\alpha=1/p_{0}-1/p_{1}.$ Then, 
%\begin{align*}
%K(f,t; L^{p_{0},q_{0}}(I^{n}), L^{\infty}(I^{n}))\approx \biggl(\int^{t^{p_{0}}}_{0}\bigl[s^{1/p_{0}-1/q_{0}}f^{*}(s)\bigr]ds\biggr)^{1/q_{0}},\ \ \textrm{for $t>0.$}
%\end{align*}
%\end{theorem}

\section{Mixed norm spaces}\label{mixed norm spaces}
In what follows and throughout the paper we shall assume $n\geq 2.$

Our goal in this section is to present some basic properties of mixed norm spaces.
Let $k\in\left\{1,\ldots,n\right\}.$ We write $\widehat{x_{k}}$  for the point in $I^{n-1}$ obtained from a given vector $x\in I^{n}$ by removing its $k\textnormal{th}$ coordinate. That is,
\begin{align*}
\widehat{x_{k}}=(x_{1},\ldots,x_{k-1},x_{k+1},\ldots,x_{n})\in I^{n-1}.
\end{align*}
Moreover, for any $f\in\mathcal{M}(I^{n}),$ we use the notation $f_{\widehat{x_{k}}}$  for the function  obtained from $f,$ with   $\widehat{x_{k}}$ fixed. Let us  recall  that, since $f$ is measurable,  $f_{\widehat{x}_{k}}$ is also measurable for a.e. $\widehat{x_{k}}$.

For later purposes, let us first enumerate  some geometric properties of the projections. We refer to the book  \cite{Halmos} for  basic facts on this topic. 

Let $E\subset I^{n}$ be any measurable set and let  $\widehat{x_{k}}\in I^{n-1}$ be fixed.  The  $x_{k}\textnormal{-section}$ of $E$   is defined as
\begin{align*}
E(\widehat{x_{k}})=\left\{x_{k}\in I: (\widehat{x_{k}},x_{k})\in E\right\}.
\end{align*}
Let us    emphasize that since $E$ is measurable, its $x_{k}\textnormal{-section}$ is also measurable, for a.e. $\widehat{x_{k}}$. The essential projection of $E$ onto the hyperplane $x_{k}=0$ is defined as
\begin{align*}
\Pi^{*}_{k}E=\left\{\widehat{x}_{k}\in I^{n-1}: |E(\widehat{x_{k}})|>0\right\}.
\end{align*}

An important result  is the so-called Loomis-Whitney inequality \cite{Loomis-Whitney, Fournier, Robert-Viktor} which says that
\begin{align}\label{eq: version Loomis-Whitney}
\bigl|E\bigr|\leq \prod^{n}_{k=1}\bigl|\Pi^{*}_{k} E\bigr|^{1/(n-1)}.
\end{align}

We now recall the Benedek-Panzone spaces, which were introduced in~\cite{Benedek-Panzone}.  For further information on this topic see \cite{Buhvalov, Blozinski, Boccuto-Bukhvalov-Sambucini,Barza-Kaminska-Persson-Soria}.

\begin{definition}\label{definition: Benedek-Panzone} Let $k\in \left\{1,\ldots,n\right\}.$  Given two r.i. spaces $X(I^{n-1})$ and  $Y(I),$ the  Benedek-Panzone space $\mathcal{R}_{k}(X,Y)$ is defined  as the collection of all $f\in\mathcal{M}(I^{n})$ satisfying
\begin{align*}
\bigl\|f\bigr\|_{\mathcal{R}_{k}(X,Y)}=\bigl\|\psi_{k}(f,Y)\bigr\|_{X(I^{n-1})}<\infty,
\end{align*}
where $\psi_{k}(f,Y)(\widehat{x}_{k})=\bigl\|f(\widehat{x}_{k},\cdot)\bigr\|_{Y(I)}.$
\end{definition}

Buhvalov \cite{Buhvalov} and  Blozinski \cite{Blozinski} proved that   $\mathcal{R}_{k}(X,Y)$  is a Banach function space. Moreover Boccuto, Bukhvalov, and Sambucini \cite{Boccuto-Bukhvalov-Sambucini} proved that $\mathcal{R}_{k}(X,Y)$ is an r.i. space, if and only if $X=Y=L^p$.

Now, we shall give the definition of the  mixed norm spaces, sometimes also called symmetric mixed norm spaces.

\begin{definition}\label{definition: mixed norm spaces}Given two r.i. spaces  $X(I^{n-1})$ and  $Y(I),$ the mixed norm space $\mathcal{R}(X,Y)$ is defined as
\begin{align*}
\mathcal{R}(X,Y)=\bigcap^{n}_{k=1}\mathcal{R}_{k}(X,Y).
\end{align*}
For each $f\in\mathcal{R}(X,Y),$ we set  $\bigl\|f\bigr\|_{\mathcal{R}(X,Y)}=\sum^{n}_{k=1}\bigl\|f\bigr\|_{\mathcal{R}_{k}(X,Y)}.$
\end{definition}
 
It is not difficult to verify that  $\mathcal{R}(X,Y)$ is a Banach function space. 

Since the pioneering works of Gagliardo~\cite{Gagliardo}, Nirenberg \cite{Nirenberg}, and Fournier \cite{Fournier}, many useful properties and generalizations of these spaces have been  studied, via different methods, by various authors, including Blei  \cite{Blei-Fournier}, Milman~\cite{Milman-kfunctional-mixed-norm}, Algervik and Kolyada \cite{Robert-Viktor}, and Kolyada \cite{Viktor2012, Viktor2013}.

All these works, together with the embedding \eqref{eq: Gagliardo-Niremberg}, provide us a strong motivation to better understand the mixed norm spaces of the form $\mathcal{R}(X,L^{\infty}).$ For this,  we start with a useful lemma:
\begin{lemma}\label{lemma: essential projection property} Let $f\in\mathcal{M}(I^{n})$ and let $E_{\alpha}=\left\{x\in I^{n}: \left|f(x)\right|>\alpha\right\},$ with $\alpha\geq 0.$ Then,
\begin{align*}
\Pi^{*}_{k}E_{\alpha}=\left\{\widehat{x}_{k}\in I^{n-1}\:: \psi_{k}(f,L^{\infty})(\widehat{x}_{k})>\alpha\right\}.
\end{align*}
\end{lemma}
\begin{proof}
To prove this lemma, it is enough to consider $\alpha>0.$  Let us see that
\begin{align}\label{eq: 1essential projection property}
\left\{\widehat{x}_{k}\in I^{n-1}\:: \psi_{k}(f,L^{\infty})(\widehat{x}_{k})>\alpha\right\}\subset \Pi^{*}_{k}E_{\alpha}.
\end{align}
In fact, if $\widehat{x}_{k}\notin \Pi^{*}_{k}E_{\alpha},$ then, by definition of $\Pi^{*}_{k}E_{\alpha},$ we have
$$\left|\left\{x_{k}\in I\::\left|f(\widehat{x}_{k},x_{k})\right|>\alpha\right\}\right|=0.$$
But
\begin{align*}
\psi_{k}(f,L^{\infty})(\widehat{x}_{k})=\inf\left\{s\geq 0\:: \left|\left\{x_{k}\in I\::\left|f(\widehat{x}_{k},x_{k})\right|>s\right\}\right|=0 \right\},\ \ \text{a.e.}\ \widehat{x_{k}}\in I^{n-1},
\end{align*}
 and hence, we get  $\psi_{k}(f,L^{\infty})(\widehat{x}_{k})\leq \alpha.$ As a consequence, we have
$$\widehat{x}_{k}\notin \left\{\widehat{x}_{k}\in I^{n-1}\:: \psi_{k}(f,L^{\infty})(\widehat{x}_{k})>\alpha\right\}.$$
 This proves that  \eqref{eq: 1essential projection property}  holds. To complete the proof, it only remains to see that 
\begin{align*}
\Pi^{*}_{k}E_{\alpha}\subset \left\{\widehat{x}_{k}\in I^{n-1}\:: \psi_{k}(f,L^{\infty})(\widehat{x}_{k})>\alpha\right\}.
\end{align*}
 In fact, if $\widehat{x}_{k}\in \Pi^{*}_{k}E_{\alpha},$ then $
\left|\left\{x_{k}\in I\::\left|f(\widehat{x}_{k},x_{k})\right|>\alpha\right\}\right|>0.$
 So, by definition,  $\psi_{k}(f,L^{\infty})(\widehat{x}_{k})>\alpha.$
Therefore, $\widehat{x}_{k}\in \left\{\widehat{x}_{k}\in I^{n-1}\:: \psi_{k}(f,L^{\infty})(\widehat{x}_{k})>\alpha\right\}.$ 
Thus, the proof is complete. 
\end{proof}

 As an immediate consequence of inequality \eqref{eq: version Loomis-Whitney}  and Lemma \ref{lemma: essential projection property}, we have  the following  inequality, which was previously  proved in \cite{Fournier}: 
\begin{corollary}\label{corollary: distribution and essential projection } Let  $f\in\mathcal{M}(I^{n}).$ Then, for any $t>0,$
\begin{align*}
\lambda_{f}(t)\leq \Bigl(\prod^{n}_{k=1}\lambda_{\psi_{k}(f,L^{\infty})}(t)\Bigr)^{1/(n-1)}.
\end{align*}
\end{corollary}

Our next goal is to describe the $K$-functional for  pairs of the form $(\mathcal{R}(X,L^{\infty}),L^{\infty}).$ Then, we shall apply it  to the characterization of the
r.i. hull of a mixed norm space of the form $\mathcal{R}(X,L^{\infty})$ (see Theorem~\ref{teo: rango R(X,Loo)->Z}). We refer to the paper \cite{Milman-kfunctional-mixed-norm}, for other results concerning  interpolation of mixed norm spaces.

We first prove a lower bound  for the  $K$-functional for  the couple of mixed norm spaces $(\mathcal{R}(X,Y),\mathcal{R}(L^{\infty},Y)).$
\begin{lemma}\label{lemma: lower bound for K-functional} Let $X(I^{n-1})$ and $Y(I)$ be r.i. spaces. Then,
\begin{align*}
\sum^{n}_{k=1}\bigl\|\psi^{*}_{k}(f,Y)\chi_{(0,t)}\bigr\|_{\overline{X}(0,|I|^{n-1})}\lesssim K\bigl(f,\varphi_{X}(t), \mathcal{R}(X,Y),\mathcal{R}(L^{\infty},Y)\bigr),\ \ 0<t<|I|^{n-1}.
\end{align*}
\end{lemma}

\begin{proof}
We fix $0<t<|I|^{n-1}$  and  $k\in \left\{1,\ldots,n\right\}.$ If $f=f_{0}+f_{1},$ with $f_{0}\in\mathcal{R}(X,Y)$ and $f_{1}\in \mathcal{R}(L^{\infty},Y),$ then 
\begin{align*}
\psi_{k}(f,Y)(\widehat{x_{k}})\leq \psi_{k}(f_{0},Y)(\widehat{x_{k}})+\psi_{k}(f_{1},Y)(\widehat{x_{k}}),\ \ \widehat{x_{k}}\in I^{n-1}.
\end{align*}
So,  it holds that
\begin{align*}
\psi^{*}_{k}(f,Y)(t)\leq \psi^{*}_{k}(f_{0},Y)(t)+\psi^{*}_{k}(f_{1},Y)(0)=\psi^{*}_{k}(f_{0},Y)(t)+\bigl\|f_{1}\bigr\|_{\mathcal{R}_{k}(L^{\infty},Y)}.
\end{align*}
Therefore,  we have
\begin{align*}
\bigl\|\psi^{*}_{k}(f,Y)\chi_{(0,t)}\bigr\|_{\overline{X}(0,|I|^{n-1})}&\leq \bigl\|\psi^{*}_{k}(f_{0},Y)\chi_{(0,t)}\bigr\|_{\overline{X}(0,|I|^{n-1})}+\varphi_{X}(t)\bigl\|f_{1}\bigr\|_{\mathcal{R}_{k}(L^{\infty},Y)}\\
& \leq \bigl\|f_{0}\bigr\|_{\mathcal{R}(X,Y)}+\varphi_{X}(t)\bigl\|f_{1}\bigr\|_{\mathcal{R}(L^{\infty},Y)}.
\end{align*}
Hence, taking the infimum over all decompositions of $f$ of the form $f=f_{0}+f_{1},$ with $f_{0}\in\mathcal{R}(X,Y)$ and $f_{1}\in \mathcal{R}(L^{\infty},Y),$ we obtain
\begin{align*}
\bigl\|\psi^{*}_{k}(f,Y)\chi_{(0,t)}\bigr\|_{\overline{X}(0,|I|^{n-1})}&\leq K(f,\varphi_{X}(t), \mathcal{R}(X,Y),\mathcal{R}(L^{\infty},Y)),
\end{align*}
for any $k\in \bigl\{1,\ldots,n\bigr\}$ and $0<t<|I|^{n-1}.$ Consequently,
\begin{align*}
\sum^{n}_{k=1}\bigl\|\psi^{*}_{k}(f,Y)\chi_{(0,t)}\bigr\|_{\overline{X}(0,|I|^{n-1})}\leq n  K(f,\varphi_{X}(t), \mathcal{R}(X,Y),\mathcal{R}(L^{\infty},Y)),\ \ 0<t<|I|^{n-1},
\end{align*}
from which the result follows.
\end{proof} 
 
\begin{theorem}\label{theorem: K-functional of R(X,Loo) and Loo} Let $X(I^{n-1})$ be an r.i. space and let   $f\in  \mathcal{R}(X,L^{\infty})+L^{\infty}(I^{n}).$ Then, 
\begin{align*}
K(f,\varphi_{X}(t), \mathcal{R}(X,L^{\infty}),L^{\infty})\approx\sum^{n}_{k=1}\bigl\|\psi^{*}_{k}(f,L^{\infty})\chi_{(0,t)}\bigr\|_{\overline{X}(0,|I|^{n-1})},\ \ 0<t<|I|^{n-1},
\end{align*}
where $\varphi_{X}(t)$ is the fundamental function of $X(I^{n-1}).$ 
\end{theorem}
\begin{proof}
In view of Lemma~\ref{lemma: lower bound for K-functional}, we only need to prove
\begin{align*}
K(f,\varphi_{X}(t), \mathcal{R}(X,L^{\infty}),L^{\infty})\lesssim \sum^{n}_{k=1}\bigl\|\psi^{*}_{k}(f,L^{\infty})\chi_{(0,t)}\bigr\|_{\overline{X}(0,|I|^{n-1})},\ \ 0<t<|I|^{n-1}.
\end{align*}
 For this, we fix any $0<t<|I|^{n-1}.$ Then, we define 
\begin{align*}
\alpha_{t}=\sum^{n}_{j=1}\psi^{*}_{j}(f,L^{\infty})(t),
\end{align*}

\begin{align*}
F(x)=\begin{cases}
\displaystyle{f(x)-\frac{\alpha_{t}f(x)}{|f(x)|}},&\textnormal{if $x\in A_{t}= \left\{x\in I^{n}: \left|f(x)\right|>\alpha_{t}\right\},$ }\\
0,& \textnormal{otherwise,}
\end{cases}
\end{align*}
and $G=f-F.$ Let $k\in \left\{1,\ldots,n\right\}$  be fixed. Then, for any $\widehat{x_{k}}\in I^{n-1},$
\begin{align*}
F_{\widehat{x_{k}}}(y)=\begin{cases}
\displaystyle{f_{\widehat{x_{k}}}(y)-\frac{\alpha_{t}f(\widehat{x_{k}})(y)}{|f_{\widehat{x_{k}}}(y)|}},&\textnormal{if $y\in A_{t}(\widehat{x_{k}}),$ }\\
0,& \textnormal{otherwise,}
\end{cases}
\end{align*}
where
\begin{align*}
A_{t}(\widehat{x_{k}})= \left\{y\in I: (\widehat{x_{k}},y)\in A_{t}\right\}=\left\{y\in I: |f_{\widehat{x_{k}}}(y)|>\alpha_{t}\right\}.
\end{align*}
Thus, for any $s\geq 0$ and $\widehat{x_{k}}\in I^{n-1},$
\begin{align*}
\lambda_{F_{\widehat{x_{k}}}}(s)&=\left|\left\{y\in I: |F_{\widehat{x_{k}}}(y)|>s\right\}\right|=\left|\left\{y\in A_{t}(\widehat{x_{k}}): \bigl||f_{\widehat{x_{k}}}(y)|-\alpha_{t}\bigr|>s\right\}\right|\\
&=\left|\left\{y\in I: |f_{\widehat{x_{k}}}(y)|-\alpha_{t}>s\right\}\right|=\lambda_{f_{\widehat{x_{k}}}}(s+\alpha_{t}).
\end{align*}
Now, let us suppose that $\widehat{x_{k}}\not\in \Pi^{*}_{k}A_{t}.$ Then, by definition, $\lambda_{f_{\widehat{x_{k}}}}(\alpha_{t})=0.$ As a consequence, we have that if $\widehat{x_{k}}\not\in \Pi^{*}_{k}A_{t},$ then $\lambda_{F_{\widehat{x_{k}}}}(s)=0,$ for any $s\geq 0.$ Therefore, for any $s\geq 0,$ it holds that
\begin{align*}
\lambda_{F_{\widehat{x_{k}}}}(s)&=\begin{cases}\lambda_{f_{\widehat{x_{k}}}}(s+\alpha_{t}),& \textnormal{if $\widehat{x_{k}}\in \Pi^{*}_{k}A_{t},$}\\
0,&\textnormal{otherwise.}\end{cases}
\end{align*}
So, Lemma~\ref{lemma: essential projection property} implies that, for any $s\geq 0,$
\begin{align*}
\lambda_{F_{\widehat{x_{k}}}}(s)&=\begin{cases}
\lambda_{f_{\widehat{x_{k}}}}(s+\alpha_{t}),&\textnormal{if $\widehat{x_{k}}\in \left\{\widehat{x_{k}}\in I^{n-1}: \psi_{k}(f,L^{\infty})(\widehat{x_{k}})>\alpha_{t}\right\},$}\\
0,&\textnormal{otherwise.}
\end{cases}
\end{align*}
Hence, for any $\widehat{x_{k}}\in \left\{\widehat{x_{k}}\in I^{n-1}: \psi_{k}(f,L^{\infty})(\widehat{x_{k}})>\alpha_{t}\right\},$ we get
\begin{align*}
\psi_{k}(F,L^{\infty})(\widehat{x_{k}})&=\inf\left\{y>0: \lambda_{F_{x_{k}}}(y)=0\right\}=\inf\left\{y>0: \lambda_{f_{x_{k}}}(y+\alpha_{t})=0\right\}\\
&=\psi_{k}(f,L^{\infty})(\widehat{x_{k}})-\alpha_{t}.
\end{align*}
Therefore,  we obtain
\begin{align*}
\bigl\|F\bigr\|_{\mathcal{R}_{k}(X,L^{\infty})}&=\bigl\|\bigl(\psi_{k}(f,L^{\infty})(\widehat{x_{k}})-\alpha_{t}\bigr)\chi_{\left\{\widehat{x_{k}}\in I^{n-1}: \psi_{k}(f,L^{\infty})(\widehat{x_{k}})>\alpha_{t}\right\}}(\widehat{x_{k}})\bigr\|_{X(I^{n-1})}\\
&=\bigl\|\bigl(\psi^{*}_{k}(f,L^{\infty})-\alpha_{t}\bigr)\chi_{\bigl(0,\lambda_{\psi_{k}(f,L^{\infty})}(\alpha_{t})\bigr)}\bigr\|_{\overline{X}(0,|I|^{n-1})}.
\end{align*}
But, by hypothesis,  $\alpha_{t}=\sum^{n}_{j=1}\psi^{*}_{j}(f,L^{\infty})(t)$ and hence  we get
\begin{align*}
\bigl\|F\bigr\|_{\mathcal{R}_{k}(X,L^{\infty})}&\leq\bigl\|\psi^{*}_{k}(f,L^{\infty})\chi_{(0,t)}\bigr\|_{\overline{X}(0,|I|^{n-1})},\ \ \textnormal{for any $k\in \left\{1,\ldots,n\right\}.$}
\end{align*}
Therefore, using the above inequality, we obtain
\begin{align*}
K(f,\varphi_{X}(t), \mathcal{R}(X,L^{\infty}),L^{\infty})&\leq \bigl\|F\bigr\|_{\mathcal{R}(X,L^{\infty})}+\varphi_{X}(t)\bigl\|G\bigr\|_{L^{\infty}(\mathbb R^{n})}\\
&=\sum^{n}_{k=1}\bigl\|F\bigr\|_{\mathcal{R}_{k}(L^{1},L^{\infty})}+ \varphi_{X}(t)\alpha_{t}\\
&\leq \sum^{n}_{k=1}\bigl\|\psi^{*}_{k}(f,L^{\infty})\chi_{(0,t)}\bigr\|_{\overline{X}(0,|I|^{n-1})}+ \varphi_{X}(t)\alpha_{t}.
\end{align*}
But, it holds that 
\begin{align*}
\varphi_{X}(t)\alpha_{t}\leq \sum^{n}_{k=1}\bigl\|\psi^{*}_{k}(f,L^{\infty})\chi_{(0,t)}\bigr\|_{\overline{X}(0,|I|^{n-1})},
\end{align*}
and hence, we have
\begin{align*}
K(f,\varphi_{X}(t), \mathcal{R}(X,L^{\infty}),L^{\infty})&\leq 2\sum^{n}_{k=1}\bigl\|\psi^{*}_{k}(f,L^{\infty})\chi_{(0,t)}\bigr\|_{\overline{X}(0,|I|^{n-1})},\ \ t>0.
\end{align*}
 Thus, the proof is complete.
\end{proof}
 
 As a consequence we obtain the following result. Recall that $(A,B)_{\theta,q}$ stands for the real interpolation space of the couple $(A,B)$  \cite[Definition~V.1.7]{Bennett}:
 
\begin{corollary}\label{corollary: mixtas e interpolacion} Let $X(I^{n-1})$ be an r.i. space. Let $0<\theta<1$ and $1\leq q\leq \infty.$  Then, 
$$(\mathcal{R}(X,L^{\infty}),L^{\infty})_{\theta,q}=\mathcal{R}((X,L^{\infty})_{\theta,q},L^{\infty}), $$ 
with equivalent norms.
\end{corollary}
To  prove it, we first need to recall a result concerning the $K$-functional of pairs of r.i. spaces. For further information see \cite{Milman-kfunctioanl, Arazy}.  
\begin{theorem}\label{theorem: K-functional of r.i. spaces} Let $X(I^{n})$ be an r.i. space. Then, 
\begin{align*}
K(f,\varphi_{X}(t), X ,L^{\infty})\approx \bigl\|f^{*}\chi_{(0,t)}\bigr\|_{\overline{X}(0,|I|^{n})},\ \ t>0.
\end{align*}
\end{theorem}
\begin{proof}[Proof of Corollary~\ref{corollary: mixtas e interpolacion}]For the sake of simplicity, we prove this result only when $1\leq q<\infty.$  Let $f\in (\mathcal{R}(X,L^{\infty}),L^{\infty})_{\theta,q}.$  Then, by a change of variables, we get
\begin{align*}
\bigl\|f\bigr\|^{q}_{(\mathcal{R}(X,L^{\infty}),L^{\infty})_{\theta,q}}&=\int^{\infty}_{0}t^{-\theta q-1}\bigl[K(f,t, \mathcal{R}(X,L^{\infty}),L^{\infty})\bigr]^{q}dt\\
&=\int^{\infty}_{0}\bigl(\varphi_{X}(s)\bigr)^{-\theta q-1}\bigl[K(f,\varphi_{X}(s), \mathcal{R}(X,L^{\infty}),L^{\infty})\bigr]^{q}d\varphi_{X}(s).
\end{align*}
Hence, using Theorem \ref{theorem: K-functional of R(X,Loo) and Loo}, we obtain
\begin{align*}
\bigl\|f\bigr\|^{q}_{(\mathcal{R}(L^{1},L^{\infty}),L^{\infty})_{\theta,q}}&\approx
\int^{\infty}_{0}\bigl(\varphi_{X}(s)\bigr)^{-\theta q-1}\Bigl[\sum^{n}_{k=1}\bigl\|\psi^{*}_{k}(f,L^{\infty})\chi_{(0,s)}\bigr\|_{\overline{X}(0,|I|^{n-1})}\Bigr]^{q}ds\\
&\geq \int^{\infty}_{0}\bigl(\varphi_{X}(s)\bigr)^{-\theta q-1}\bigl\|\psi^{*}_{k}(f,L^{\infty})\chi_{(0,s)}\bigr\|^{q}_{\overline{X}(0,|I|^{n-1})}ds,
\end{align*}
for any $k\in \left\{1,\ldots,n\right\}$. So, Theorem \ref{theorem: K-functional of r.i. spaces} implies that
\begin{align*}
\bigl\|f\bigr\|^{q}_{(\mathcal{R}(L^{1},L^{\infty}),L^{\infty})_{\theta,q}}&\gtrsim \int^{\infty}_{0}\bigl(\varphi_{X}(s)\bigr)^{-\theta q-1}\bigl[K(\psi_{k}(f,L^{\infty}),\varphi_{X}(s), X,L^{\infty})\bigr]^{q}d\varphi_{X}(s)\\
&=\bigl\|\psi_{k}(f,L^{\infty})\bigr\|^{q}_{(X,L^{\infty})_{\theta,q}}=\bigl\|f\bigr\|_{\mathcal{R}_{k}((X,L^{\infty})_{q,\theta},L^{\infty})}.
\end{align*}
 As a consequence, we get
\begin{align*}
\bigl\|f\bigr\|_{\mathcal{R}((X,L^{\infty})_{\theta,q},L^{\infty})}=\sum^{n}_{k=1}\bigl\|f\bigr\|_{\mathcal{R}((X,L^{\infty})_{\theta,q},L^{\infty})}\lesssim \bigl\|f\bigr\|_{(\mathcal{R}(X,L^{\infty}),L^{\infty})_{\theta,q}}.
\end{align*}
Thus,  we have seen that the embedding 
\begin{align*}
(\mathcal{R}(X,L^{\infty}),L^{\infty})_{\theta,q}\hookrightarrow \mathcal{R}((X,L^{\infty})_{\theta,q},L^{\infty})
\end{align*}
holds. Hence, to complete the proof, it only remains to see that 
\begin{align*}
 \mathcal{R}((X,L^{\infty})_{\theta,q},L^{\infty})\hookrightarrow (\mathcal{R}(X,L^{\infty}),L^{\infty})_{\theta,q},
\end{align*}
also holds. To do it, we fix any $f\in  \mathcal{R}((X,L^{\infty})_{\theta,q},L^{\infty}).$ Then,   using Theorem~\ref{theorem: K-functional of R(X,Loo) and Loo} and the subadditive property of $\bigl\|\cdot\bigr\|_{(\mathcal{R}(X,L^{\infty}),L^{\infty})_{\theta,q}},$ we get
\begin{align*}
\bigl\|f\bigr\|_{(\mathcal{R}(X,L^{\infty}),L^{\infty})_{\theta,q}}\leq \sum^{n}_{k=1} \Bigl(\int^{\infty}_{0}\bigl(\varphi_{X}(s)\bigr)^{-\theta q-1}\bigl\|\psi^{*}_{k}(f,L^{\infty})\chi_{(0,s)}\bigr\|^{q}_{\overline{X}(0,|I|^{n-1})}d\varphi_{X}(s)\Bigr)^{1/q}.
\end{align*}
So, using Theorem \ref{theorem: K-functional of r.i. spaces} , we obtain
\begin{align*}
\bigl\|f\bigr\|_{(\mathcal{R}(L^{1},L^{\infty}),L^{\infty})_{\theta,q}}&\lesssim\sum^{n}_{k=1}\bigl\|\psi_{k}(f,L^{\infty})\bigr\|_{(X,L^{\infty})_{\theta,q}}=\bigl\|f\bigr\|_{\mathcal{R}((X,L^{\infty})_{\theta,q},L^{\infty})}.
\end{align*}
That is, $\mathcal{R}((X,L^{\infty})_{\theta,q},L^{\infty})\hookrightarrow(\mathcal{R}(X,L^{\infty}),L^{\infty})_{\theta,q}.$ Thus, the proof is complete.
\end{proof}

\section{Embeddings between mixed norm spaces}\label{section: Embeddings between mixed norm spaces}

Our aim in this section  is to  characterize certain embeddings between mixed norm spaces. 
Before that, let us   emphasize that relations between mixed norm  spaces of Lorentz spaces were studied in \cite{Robert-Viktor}, where it was shown, for instance, 
 \begin{align}\label{eq: viktor-robert}
 \mathcal{R}(L^{1},L^{\infty})\hookrightarrow \mathcal{R}(L^{(n-1)^{\prime},1},L^{1}),\ \ n\geq 2.
 \end{align}
 
Let us start with some preliminary lemmas:

\begin{lemma}\label{lema: inclusion the Benedek-Panzone spaces} Let $k\in\bigl\{1,\ldots,n\bigr\}.$  Let $X_{1}(I^{n-1}),$ $X_{2}(I^{n-1}),$ $Y_{1}(I),$ and  $Y_{2}(I)$ be r.i. spaces. Then, 
\begin{align*}
\mathcal{R}_{k}(X_{1},Y_{1})\hookrightarrow \mathcal{R}_{k}(X_{2},Y_{2})\Leftrightarrow \begin{cases}X_{1}(I^{n-1})\hookrightarrow X_{2}(I^{n-1}),\\ Y_{1}(I)\hookrightarrow Y_{2}(I).\end{cases}
\end{align*}
\end{lemma}
\begin{proof}To prove the implication ``$\Rightarrow$", we just have to apply the hypothesis to the functions 
\begin{align*}
g_{1}(x)=f_{1}(\widehat{x_{k}})\chi_{I^{n}}(x) \ \ \textnormal{and}\ \  g_{2}(x)=f_{2}(x_{k})\chi_{I^{n}}(x),
\end{align*}
 with $f_{1}\in X_{1}(I^{n-1})$ and $f_{2}\in Y_{1}(I)$. The converse  follows   from   Definition~\ref{definition: Benedek-Panzone}.
\end{proof}

It is important to observe that there are examples, for instance \eqref{eq: viktor-robert}, showing that if in Lemma~\ref{lema: inclusion the Benedek-Panzone spaces} we replace the  Benedek-Panzone spaces by  mixed norm spaces, then the corresponding equivalence is not longer true. However, we always have this result:

\begin{lemma}\label{lema: inclusion norma mixta} Let $X_{1}(I^{n-1}),$ $X_{2}(I^{n-1}),$ $Y_{1}(I),$ and  $Y_{2}(I)$ be r.i. spaces. Then, 
$$
X_{1}(I^{n-1})\hookrightarrow X_{2}(I^{n-1})\text{ and } Y_{1}(I)\hookrightarrow Y_{2}(I)
\Rightarrow \mathcal{R}(X_{1},Y_{1})\hookrightarrow \mathcal{R}(X_{2},Y_{2}).
$$
\end{lemma}

\begin{proof}
It is immediate from   Definition~\ref{definition: mixed norm spaces} and Lemma~\ref{lema: inclusion the Benedek-Panzone spaces}.
\end{proof}

Lemma~\ref{lema: inclusion the Benedek-Panzone spaces} and Lemma~\ref{lema: inclusion norma mixta} show that it is natural to study when embeddings  between mixed norm spaces are true.  Motivated by this problem, we shall find necessary and sufficient conditions in the following cases:
\begin{align*}
\begin{array}{lll}
\mathcal{R}(L^{\infty},Y_{1})\hookrightarrow \mathcal{R}(X_{2},Y_{2}),&  \mathcal{R}(X_{1},L^{\infty})\hookrightarrow \mathcal{R}(X_{2},L^{\infty}),&  \mathcal{R}(X_{1},Y)\hookrightarrow \mathcal{R}(X_{2},Y).
\end{array}
\end{align*}

\begin{theorem}\label{teo: mixta inclusion R(X1,Y1)->R(X2,Y2)}For any r.i. spaces  $X_{1}(I^{n-1}),$ $X_{2}(I^{n-1}),$  $Y_{1}(I)$ and $Y_{2}(I),$ if the following embedding 
\begin{align*}
\mathcal{R}(X_{1},Y_{1})\hookrightarrow \mathcal{R}(X_{2},Y_{2})
\end{align*}
holds, then $Y_{1}(I)\hookrightarrow Y_{2}(I).$ 
\end{theorem}

\begin{proof}By Lemma~\ref{lema: inclusion norma mixta}, we may assume, without loss of generality, that the following embedding
\begin{align*}
\mathcal{R}(L^{\infty},Y_{1})\hookrightarrow \mathcal{R}(L^{1},Y_{2})
\end{align*}
holds. Also, we shall suppose that  $I=(-a,b),$ with $a,b\in\mathbb R_{+}.$ Let $r\in \mathbb R$ such that  $0<r<\min(a,b).$ Given any function $g\in {Y}_{1}(I),$ with $\lambda_{g}(0)\leq 2r/n,$ we   define 
 \begin{align*}
 f(x)=g^{*}\Bigl(2\Bigl|\sum^{n}_{i=1}x_{i}\Bigr|\Bigr)\chi_{(-r,r)^{n}}(x).
 \end{align*}
For any $k\in \bigr\{1,\ldots,n\bigr\}$, we denote 
 \begin{align*}
 \beta_{k}=\sum^{n}_{i=1,i\neq k}x_{i},\ \ \ \textnormal{whenever $\widehat{x_{k}}\in (-r,r)^{n-1}$.}
 \end{align*}
  Now, if $s\geq 0,$ we have
 \begin{align*}
 \lambda_{f_{\widehat{x_{k}}}}(s)&=\left|\left\{x\in (-r,r): g^{*}(2|x+\beta_{k}|)>s\right\}\right|\\
 &=\left|\left\{x\in (-r,r)\cap (-r/n-\beta_{k}, r/n-\beta_{k}) : g^{*}(2|x+\beta_{k}|)>s\right\}\right|\\
 &\leq \left|\left\{x\in (-r/n-\beta_{k}, r/n-\beta_{k}): g^{*}(2|x+\beta_{k}|)>s\right\}\right|\\
 &= \left|\left\{x\in (-r/n, r/n): g^{*}(2|x|)>s\right\}\right|=\lambda_{g}(s).
 \end{align*} 
 Thus,
 \begin{align*}
 \left\{s\geq 0: \lambda_{g}(s)\leq t\right\}\subseteq \bigl\{s\geq 0: \lambda_{f_{\widehat{x_{k}}}}(s)\leq t\bigr\},\ \ \textnormal{for any $t\geq 0,$}
 \end{align*}
 and so $f^{*}_{\widehat{x_{k}}}\leq g^{*}.$ Hence, we get
 \begin{align*}
 \psi_{k}(f,Y_{1})(\widehat{x_{k}})&\leq \bigl\|g^{*}\bigr\|_{\overline{Y}_{1}(0,|I|)},\ \ \widehat{x_{k}}\in (-r,r)^{n-1}.
 \end{align*}
 Therefore,  
 \begin{align*}
 \bigl\|f\bigr\|_{\mathcal{R}_{k}(L^{\infty},Y_{1})}\leq \bigl\|g^{*}\bigr\|_{\overline{Y}_{1}(0,|I|)},\ \ k\in \left\{1,\ldots,n\right\}.
 \end{align*}
 Hence, our assumption on $g$ ensures that $f\in\mathcal{R}(L^{\infty},Y_{1})$ and
 \begin{align}\label{eq: 1mixta inclusion R(X1,Y1)->R(X2,Y2)}
 \bigl\|f\bigr\|_{\mathcal{R}(L^{\infty},Y_{1})}\leq n\bigl\|g^{*}\bigr\|_{\overline{Y}_{1}(0,|I|)}.
 \end{align}
Thus, using $\mathcal{R}(L^{\infty},Y_{1})\hookrightarrow\mathcal{R}(L^{1},Y_{2})$ and \eqref{eq: 1mixta inclusion R(X1,Y1)->R(X2,Y2)}, we get
  \begin{align}\label{eq: 2mixta inclusion R(X1,Y1)->R(X2,Y2)}
 \bigl\|f\bigr\|_{\mathcal{R}(L^{1},Y_{2})}\lesssim \bigl\|g^{*}\bigr\|_{\overline{Y}_{1}(0,|I|)}.
 \end{align}
 Now, let us compute $\bigl\|f\bigr\|_{\mathcal{R}(L^{1},Y_{2})}.$ In order to do it, we fix any  $k\in \left\{1,\ldots,n\right\}$ and $\widehat{x_{k}}\in (0,r/n)^{n-1},$ and set
 \begin{align*}
 \gamma_{k}=\sum^{n}_{i=1,i\neq k}x_{i}.
 \end{align*} 
 As before, if $s\geq 0,$ we have
 \begin{align*}
 \lambda_{f_{\widehat{x_{k}}}}(s)&=\left|\left\{x\in (-r,r)\cap (-r/n-\gamma_{k}, r/n-\gamma_{k}) : g^{*}(2|x+\gamma_{k}|)>s\right\}\right|.
 \end{align*}
 But, $0<\gamma_{k}<r/n^{\prime},$ so we obtain
 \begin{align*}
\lambda_{f_{\widehat{x_{k}}}}(s)=\left|\left\{x\in (-r/n-\gamma_{k},r/n-\gamma_{k}): g^{*}(2|x+\gamma_{k}|)>s\right\}\right|=\lambda_{g}(s),
 \end{align*} 
 for any $s\geq 0.$ As a consequence, if $\widehat{x_{k}}\in (0,r/n)^{n-1},$ then $f^{*}_{\widehat{x_{k}}}=g^{*}.$ Thus, 
 \begin{align*}
\psi_{k}(f,Y_{2})(\widehat{x_{k}})&=\bigl\|f(\widehat{x_{k}},\cdot)\bigr\|_{Y_{2}(I)}\chi_{(-r,r)^{n-1}}(\widehat{x_{k}})
\geq \bigl\|f(\widehat{x_{k}},\cdot)\bigr\|_{Y_{2}(I)}\chi_{ (0,r/n)^{n-1}}(\widehat{x_{k}})\\
&=\bigl\|g^{*}\bigr\|_{\overline{Y}_{2}(0,|I|)}\chi_{ (0,r/n)^{n-1}}(\widehat{x_{k}}),
 \end{align*}
 and so 
 \begin{align*}
 \bigl\|f\bigr\|_{\mathcal{R}(L^{1},Y_{2})}\gtrsim \bigl\|g^{*}\bigr\|_{\overline{Y}_{2}(0,|I|)}.
 \end{align*}
 Therefore,  inequality \eqref{eq: 2mixta inclusion R(X1,Y1)->R(X2,Y2)} gives us that
 \begin{align}\label{eq: 3mixta inclusion R(X1,Y1)->R(X2,Y2)}
 \bigl\|g^{*}\bigr\|_{\overline{Y}_{2}(0,|I|)}\lesssim \bigl\|g^{*}\bigr\|_{\overline{Y}_{1}(0,|I|)},
 \end{align}
 for any $g\in Y(I),$ with $\lambda_{g}(0)\leq 2r/n.$ Now, let us consider any $g\in Y_{1}(I).$ We define
\begin{align*}
g_{1}(x)=\max\bigl(|g(x)|-g^{*}(2r/n),0\bigr)\,\textnormal{sgn}\,g(x),
\end{align*}
and 
\begin{align*}
g_{2}(x)=\min\bigl(|g(x)|,g^{*}(2r/n)\bigr)\,\textnormal{sgn}\,g(x).
\end{align*}
Since $\lambda_{g_{1}}(0)\leq 2r/n,$ the inequality \eqref{eq: 3mixta inclusion R(X1,Y1)->R(X2,Y2)}, with $g$ replaced by $g_{1},$ implies that 
\begin{align}\label{eq: 4mixta inclusion R(X1,Y1)->R(X2,Y2)}
\left\|g_{1}\right\|_{Y_{2}(I)}\lesssim \left\|g_{1}\right\|_{Y_{1}(I)}.
\end{align}
Thus,  combining $g_{1}\leq g$ a.e. with \eqref{eq: 4mixta inclusion R(X1,Y1)->R(X2,Y2)}, we get
\begin{align}\label{eq: 5mixta inclusion R(X1,Y1)->R(X2,Y2)}
\left\|g_{1}\right\|_{Y_{2}(I)}\lesssim  \left\|g\right\|_{Y_{1}(I)}.
\end{align}
 On the other hand, by H\"older's inequality, we obtain
\begin{align}\label{eq: 6mixta inclusion R(X1,Y1)->R(X2,Y2)}
\left\|g_{2}\right\|_{Y_{2}(I)}&\leq \varphi_{Y_{2}}(|I|)  g^{**}(2r/n)\lesssim \left\|g\right\|_{Y_{1}(I)}.
\end{align}
Finally, using \eqref{eq: 5mixta inclusion R(X1,Y1)->R(X2,Y2)} and \eqref{eq: 6mixta inclusion R(X1,Y1)->R(X2,Y2)},  we get
\begin{align*} 
\left\|g\right\|_{Y_{2}(I)}=\left\|g_{1}+g_{2}\right\|_{Y_{2}(I)} \lesssim \left\|g\right\|_{Y_{1}(I)},\ \ \ f\in Y_{1}(I), 
\end{align*}
and the proof is complete. 
\end{proof}

As a consequence we have the following corollary: 
\begin{corollary} Let  $X_{2}(I^{n-1}),$ $Y_{1}(I)$ and $Y_{2}(I)$ be r.i. spaces. Then, 
\begin{align*}
\mathcal{R}(L^{\infty},Y_{1})\hookrightarrow\mathcal{R}(X_{2},Y_{2})\Leftrightarrow Y_{1}(I)\hookrightarrow Y_{2}(I).
\end{align*}
\end{corollary}
\begin{proof}It follows from Lemma~\ref{lema: inclusion norma mixta} and Theorem~\ref{teo: mixta inclusion R(X1,Y1)->R(X2,Y2)}.
\end{proof}

Another consequence of Theorem~\ref{teo: mixta inclusion R(X1,Y1)->R(X2,Y2)} is the following result:
\begin{corollary}  Let $1<p_{1}\,,p_{3}<\infty,$ $1\leq q_{1}\,,q_{3}\leq \infty$ and either $p_{2}=q_{2}=1,$ $p_{2}=q_{2}=\infty$ or $1<p_{2}<\infty$ and $1\leq q_{2}\leq \infty.$ Then, 
\begin{align*} 
\mathcal{R}(L^{\infty},L^{p_{1},q_{1}})\hookrightarrow \mathcal{R}(L^{p_{2},q_{2}},L^{p_{3},q_{3}})\Leftrightarrow \begin{cases} p_{3}<p_{1}\,,1\leq q_{1}\,,q_{3}\leq \infty,\\
p_{1}=p_{3}\,, 1\leq q_{1}\leq q_{3}\leq \infty.\end{cases}\end{align*}
\end{corollary}
\begin{proof}It follows from  Theorem~\ref{teo: mixta inclusion R(X1,Y1)->R(X2,Y2)} and the classical embeddings for Lorentz spaces (see \cite{Bennett}). 
\end{proof}

Let us now study  embeddings between mixed norm spaces of the form $\mathcal{R}(X,L^{\infty}).$

\begin{theorem}\label{teo: mixta inclusion R(X1,Loo)->R(X2,Loo)} Let  $X_{1}(I^{n-1})$ and $X_{2}(I^{n-1})$ be  r.i. spaces. Then,
\begin{align*}
\mathcal{R}(X_{1},L^{\infty})\hookrightarrow \mathcal{R}(X_{2},L^{\infty})\Leftrightarrow X_{1}(I^{n-1})\hookrightarrow X_{2}(I^{n-1}).
\end{align*}
\end{theorem}

\begin{proof} In view of Lemma~\ref{lema: inclusion norma mixta} , we only need to prove that if the embedding
 \begin{align*}
 \mathcal{R}(X_{1},L^{\infty})\hookrightarrow \mathcal{R}(X_{2},L^{\infty})
 \end{align*}
 holds, then $X_{1}(I^{n-1})\hookrightarrow X_{2}(I^{n}).$  As before, we assume that $I=(-a,b),$ with $a,b\in\mathbb R_{+},$ and $0<r<\min(a,b).$ Given any   $f\in X_{1}(I^{n-1}),$ with $\lambda_{f}(0)\leq \omega_{n-1}r^{n-1},$  we define 
 \begin{align*}
g(x)=\begin{cases}
f^{*}(\omega_{n-1}|x|^{n-1}),& \textnormal{if $x\in B_{n}(0,r),$}\\
0,&\textnormal{otherwise.}
\end{cases} 
\end{align*}
We fix any $k\in \left\{1,\ldots,n\right\}.$ Then, it holds that
\begin{align*}
\psi_{k}(g,L^{\infty})(\widehat{x_{k}})=\left\|g(\widehat{x_{k}},\cdot)\right\|_{L^{\infty}(I)}=f^{*}(\omega_{n-1}|\widehat{x_{k}}|^{n-1}),\ \ \textnormal{if $\widehat{x_{k}}\in B_{n-1}(0,r),$}
\end{align*}
and $\psi_{k}(g,L^{\infty})(\widehat{x_{k}})=0$ otherwise. So, for any $k\in \left\{1,\ldots,n\right\},$ we have
\begin{align*}
\left\|g\right\|_{\mathcal{R}_{k}(X_{1},L^{\infty})}=\left\|\psi_{k}(g,L^{\infty})\right\|_{X_{1}(I^{n-1})}= \left\|f\right\|_{X_{1}(I^{n-1})}.
\end{align*}
 Hence, since we are assuming that $f\in X_{1}(I^{n-1}),$ we obtain $g\in\mathcal{R}(X_{1},L^{\infty})$ and
 \begin{align*}
 \left\|g\right\|_{\mathcal{R}(X_{1},L^{\infty})}=n \left\|f\right\|_{X_{1}(I^{n-1})}.
 \end{align*}
   So, using $ \mathcal{R}(X_{1},L^{\infty})\hookrightarrow \mathcal{R}(X_{2},L^{\infty})$ and the previous inequality, we get 
\begin{align*}
\left\|g\right\|_{\mathcal{R}(X_{2},L^{\infty})}\lesssim \left\|f\right\|_{X_{1}(I^{n-1})}.
\end{align*}
 But, as before, 
\begin{align*}
\left\|g\right\|_{\mathcal{R}(X_{2},L^{\infty})}=n\left\|f\right\|_{X_{2}(I^{n-1})},
\end{align*}
hence, we have
\begin{align*}
\left\|f\right\|_{X_{2}(I^{n-1})}\lesssim \left\|f\right\|_{X_{1}(I^{n-1})}.
\end{align*}
This proves that if $f\in X_{1}(I^{n-1}),$ with $\lambda_{f}(0)\leq \omega_{n-1}r^{n-1},$ then  $f\in X_{2}(I^{n-1}).$ The rest of the proof is essentially the same as in Theorem~\ref{teo: mixta inclusion R(X1,Y1)->R(X2,Y2)}.
\end{proof}

For a general r.i. space $Y(I),$ we have a similar result assuming some conditions on $X_{1}(I^{n-1}).$

\begin{theorem}\label{theorem: R(X1,Y)->R(X2,Y)} Let  $X_{1}(I^{n-1})$ be an r.i. space , with  $\underline{\alpha}_{X_{1}}>0,$ and let $X_{2}(I^{n-1})$ and $Y(I)$ be  r.i. spaces.  Then,
\begin{align*}
\mathcal{R}(X_{1},Y)\hookrightarrow\mathcal{R}(X_{2},Y)\Leftrightarrow  X_{1}(I^{n-1})\hookrightarrow  X_{2}(I^{n-1}).
\end{align*}
\end{theorem}

\begin{proof} As before, according to Lemma~\ref{lema: inclusion norma mixta} , it suffices to prove the necessary part of this result. Also, by Theorem~\ref{teo: mixta inclusion R(X1,Loo)->R(X2,Loo)}, we assume that $Y(I)\neq L^{\infty}(I). $ Let us suppose that the embedding 
 \begin{align*}
 \mathcal{R}(X_{1},Y)\hookrightarrow \mathcal{R}(X_{2},Y)
 \end{align*}
 holds and suppose that $I=(-a,b),$ with $a,b\in\mathbb R_{+},$  and $0<r<\min(a,b).$ Given any function  $f\in X_{1}(I^{n-1}),$ with $\lambda_{f}(0)\leq \omega_{n-1}r^{n-1},$  we define 
\begin{align*}
g(x)=\begin{cases}
\displaystyle{\int^{\omega_{n-1}r^{n-1}}_{\omega_{n-1}|x|^{n-1}}\frac{f^{*}(t)}{t\varphi_{Y}(2\,(t/\omega_{n-1})^{1/(n-1)})}\,dt},& \textnormal{if $x\in B_{n}(0,r),$}\vspace{0.3cm}\\ 
0,&\textnormal{otherwise.}
\end{cases}
\end{align*} 
We fix  $k\in \left\{1,\ldots,n\right\}$ and $\widehat{x_{k}}\in B_{n-1}(0,r).$  Using now \eqref{eq: abs. cont}, we obtain 
\begin{align*}
\psi_{k}(g,Y)(\widehat{x_{k}})&\leq \psi_{k}(g,\Lambda_{\varphi_{Y}})(\widehat{x_{k}})=\psi_{k}(g,L^{\infty})(\widehat{x_{k}})\varphi_{Y}(0^{+})+\int^{|I|}_{0}g^{*}_{\widehat{x_{k}}}(t)\varphi^{\prime}_{Y}(t)dt\\
&=\int^{2((\lambda_{f}(0)/\omega_{n-1})^{1/(n-1)}-|\widehat{x_{k}}|)}_{0}\varphi^{\prime}_{Y}(t)dt\\
&\qquad\times \biggl(\int^{\omega_{n-1}r^{n-1}}_{\omega_{n-1}(t/2+|\widehat{x_{k}}|)^{n-1}}\frac{f^{*}(s)}{s\varphi_{Y}(2\,(s/\omega_{n-1})^{1/(n-1)})}\,ds\biggr).
\end{align*}
Then, Fubini's theorem gives
\begin{align*}
\psi_{k}(g,Y)(\widehat{x_{k}})&\lesssim \int^{\lambda_{f}(0)}_{\omega_{n-1}|\widehat{x_{k}}|^{n-1}}\frac{f^{*}(s)}{s\varphi_{Y}(2\,(s/\omega_{n-1})^{1/(n-1)})}\biggl(\int^{2(s/\omega_{n-1})^{1/(n-1)}-2|\widehat{x_{k}}|}_{0}\varphi^{\prime}_{Y}(t)dt\biggr)ds\\
&=\int^{\lambda_{f}(0)}_{\omega_{n-1}|\widehat{x_{k}}|^{n-1}}\frac{f^{*}(s)\varphi_{Y}\bigl(2(s/\omega_{n-1})^{1/(n-1)}-2|\widehat{x_{k}}|\bigr)}{s\varphi_{Y}(2\,(s/\omega_{n-1})^{1/(n-1)})}\,ds.\end{align*}
Hence, using that $\varphi_{Y}$ is an increasing function, we deduce that
\begin{align*}
\psi_{k}(g,Y)(\widehat{x_{k}})&\lesssim \int^{\lambda_{f}(0)}_{\omega_{n-1}|\widehat{x_{k}}|^{n-1}}\,f^{*}(s)\frac{ds}{s}.
\end{align*}
Since $\underline{\alpha}_{X_{1}}>0,$ \cite[Theorem~V.5.15]{Bennett} ensures that the integral operator 
\begin{align*}
\int^{|I|^{n-1}}_{t}f^{*}(s)\frac{ds}{s},
\end{align*}
is bounded on $\overline{X}_{1}(I^{n-1})$ and, as a consequence, we obtain
\begin{align*}
\left\|g\right\|_{\mathcal{R}_{k}(X_{1},Y)}=\left\|\psi_{k}(g,Y)\right\|_{X_{1}(I^{n-1})}\lesssim \biggl\| \int^{\lambda_{f}(0)}_{t}\,f^{*}(s)\frac{ds}{s}\biggr\|_{\overline{X}_{1}(0,|I|^{n-1})}\lesssim \left\|f\right\|_{X_{1}(I^{n-1})},
\end{align*}
for any $k\in \left\{1,\ldots,n\right\}.$ Hence, our assumption on $f$ gives that $g\in\mathcal{R}(X_{1},Y)$ and
 \begin{align}\label{eq: 1mixta inclusion R(X1,Y)->R(X2,Y)}
 \bigl\|g\bigr\|_{\mathcal{R}(X_{1},Y)}\lesssim \bigl\|f\bigr\|_{X_{1}(I^{n-1})}.
 \end{align}
 So, using $\mathcal{R}(X_{1},Y)\hookrightarrow\mathcal{R}(X_{2},Y)$ and \eqref{eq: 1mixta inclusion R(X1,Y)->R(X2,Y)}, we get
  \begin{align}\label{eq: 2mixta inclusion R(X1,Y)->R(X2,Y)}
 \bigl\|g\bigr\|_{\mathcal{R}(X_{2},Y)}\lesssim \bigl\|f\bigr\|_{X_{1}(I^{n-1})}.
 \end{align}
  We next find a lower estimate for $ \bigl\|g\bigr\|_{\mathcal{R}(X_{2},Y)}$. In fact, we fix $k\in \left\{1,\ldots,n\right\}$ and $\widehat{x_{k}}\in B_{n-1}(0,r/2).$ Then, by H\"older's inequality, we get
\begin{align}\label{eq: 3mixta inclusion R(X1,Y)->R(X2,Y)}
\frac{1}{\varphi_{Y^{\prime}}(2|\widehat{x_{k}}|)}\int_{B_{1}(0,|\widehat{x_{k}}|)}g(\widehat{x_{k}},x_{k})dx_{k}\leq \psi_{k}(g,Y)(\widehat{x_{k}}).
\end{align}
On the other hand, by   a change of variables, it holds that 
\begin{align*}
\int_{B_{1}(0,|\widehat{x_{k}}|)}g(\widehat{x_{k}},x_{k})dx_{k}\approx&\int^{2^{n-1}\omega_{n-1}|\widehat{x_{k}}|^{n-1}}_{\omega_{n-1}|\widehat{x_{k}}|^{n-1}}t^{1/(n-1)}\frac{dt}{t}\\
&\qquad\times\biggl(\int^{\omega_{n-1}r^{n-1}}_{t}\frac{f^{*}(s)}{s\varphi_{Y}(2\,(s/\omega_{n-1})^{1/(n-1)})}ds\biggr),
\end{align*}
and so Fubini's theorem gives
 \begin{align*}
 \int_{B_{1}(0,|\widehat{x_{k}}|)}g(\widehat{x_{k}},x_{k})&\gtrsim \int^{2^{n-1}\omega_{n-1}|\widehat{x_{k}}|^{n-1}}_{\omega_{n-1}|\widehat{x_{k}}|^{n-1}}\frac{f^{*}(t)}{t\varphi_{Y}(2\,(t/\omega_{n-1})^{1/(n-1)})}\,dt\\
 &\qquad\times\biggl(\int^{t}_{\omega_{n-1}|\widehat{x_{k}}|^{n-1}}s^{1/(n-1)-1}ds\biggr)\\
&\gtrsim\int^{2^{n-1}\omega_{n-1}|\widehat{x_{k}}|^{n-1}}_{(3/2)^{n-1}\omega_{n-1}|\widehat{x_{k}}|^{n-1}}\frac{f^{*}(t)\bigl(t^{1/(n-1)}-\omega^{1/(n-1)}_{n-1}|\widehat{x_{k}}|\bigr)}{t\varphi_{Y}(2\,(t/\omega_{n-1})^{1/(n-1)})}dt\\
 &\gtrsim \varphi_{Y^{\prime}}(2|\widehat{x_{k}}|)f^{*}(2^{n-1}\omega_{n-1}|\widehat{x_{k}}|^{n-1}).
 \end{align*}
 Hence, using \eqref{eq: 3mixta inclusion R(X1,Y)->R(X2,Y)}, we obtain
 \begin{align*}
 f^{*}(2^{n-1}\omega_{n-1}|\widehat{x_{k}}|^{n-1})\lesssim \psi_{k}(g,Y)(\widehat{x_{k}}), \ \ \widehat{x_{k}}\in B_{n-1}(0,r/2). 
 \end{align*}
 Therefore,
 \begin{align*}
 \bigl\|f^{*}\bigr\|_{\overline{X}_{2}(0,|I|^{n-1})}\lesssim \bigl\|\chi_{(0,\lambda_{f}(0)/2)}f^{*}(2t)\bigr\|_{\overline{X}_{2}(0,|I|^{n-1})}\lesssim  \bigl\|\psi_{k}(g,Y)\bigr\|_{X_{2}(I^{n-1})}.
 \end{align*}
 Thus, using \eqref{eq: 2mixta inclusion R(X1,Y)->R(X2,Y)}, we get
\begin{align*}
\bigl\|f\bigr\|_{X_{2}(I^{n-1})}\lesssim\bigl\|f\bigr\|_{X_{1}(I^{n-1})}.
\end{align*}
This proves that if $f\in X_{1}(I^{n-1}),$ with $\lambda_{f}(0)\leq \omega_{n-1}r^{n-1},$ then  $f\in X_{2}(I^{n-1}).$  The general case can be
treated as at the end of the proof of Theorem~\ref{teo: mixta inclusion R(X1,Y1)->R(X2,Y2)}.
\end{proof}

As a consequence  of Theorem~\ref{teo: mixta inclusion R(X1,Loo)->R(X2,Loo)} and Theorem~\ref{theorem: R(X1,Y)->R(X2,Y)}, we obtain the following result:

\begin{corollary}  Let $1<p_{1}\,,p_{3}<\infty,$ $1\leq q_{1}\,,q_{3}\leq \infty$ and either $p_{2}=q_{2}=1,$ $p_{2}=q_{2}=\infty$ or $1<p_{2}<\infty$ and $1\leq q_{2}\leq \infty.$ Then 
\begin{align*}\mathcal{R}(L^{p_{1},q_{1}},L^{p_{2},q_{2}})\hookrightarrow \mathcal{R}(L^{p_{3},q_{3}},L^{p_{2},q_{2}})\Leftrightarrow \begin{cases} p_{3}<p_{1}\,,1\leq q_{1}\,,q_{3}\leq \infty,\\
p_{1}=p_{3}\,, 1\leq q_{1}\leq q_{3}\leq \infty.\end{cases}
\end{align*}
\end{corollary}
\begin{proof}It follows from  Theorem~\ref{teo: mixta inclusion R(X1,Loo)->R(X2,Loo)}, Theorem~\ref{theorem: R(X1,Y)->R(X2,Y)} and the classical embeddings for Lorentz spaces (see \cite{Bennett}). 
\end{proof}

Finally, let us study the embedding
\begin{align}\label{eq: R(X1,Loo)->R(X2,Lp1)}
\mathcal{R}(X_{1},L^{\infty})\hookrightarrow \mathcal{R}(X_{2},L^{1}).
\end{align}
Let us start by analyzing the case  $n=2.$  The following result will be useful for our purposes.

\begin{lemma}\label{lemma: Fubini mixed norm} Let $X(I)$ be an r.i. space. Then,  
 \begin{align*}
 \mathcal{R}(L^{1},X)\hookrightarrow \mathcal{R}(X,L^{1}).
 \end{align*}
\end{lemma}

\begin{proof}Let $f\in \mathcal{R}(L^{1},X)$ and $k\in \left\{1,2\right\}.$ Then, using Fubini's theorem and H\"older's inequality, we get
\begin{align*}
\bigl\|f\bigr\|_{\mathcal{R}_{k}(X,L^{1})}&=\sup_{\bigl\|g\bigr\|_{X^{\prime}(I)}\leq 1}\int_{I}\int_{I}|g(\widehat{x_{k}})f(\widehat{x_{k}},x_{k})|d\widehat{x_{k}}dx_{k}\\
&\leq \int_{I}\psi_{k}(f,X)(\widehat{x_{k}})d\widehat{x_{k}}=\bigl\|f\bigr\|_{\mathcal{R}_{k}(L^{1},X)}\leq \bigl\|f\bigr\|_{\mathcal{R}(L^{1},X)}.
\end{align*}
That is, $\mathcal{R}(L^{1},X)\hookrightarrow \mathcal{R}(X,L^{1})$ and the proof is complete.
\end{proof}

\begin{corollary} For any couple of  r.i. spaces $X_{1}(I)$ and $X_{2}(I)$ ,  we have
 $$\mathcal{R}(X_{1},L^{\infty})\hookrightarrow \mathcal{R}(X_{2},L^{1}).$$
\end{corollary}
\begin{proof} Using Lemma~\ref{lema: inclusion norma mixta} and Lemma~\ref{lemma: Fubini mixed norm}, we get
\begin{align*}
\mathcal{R}(X_{1},L^{\infty})\hookrightarrow \mathcal{R}(L^{1},L^{\infty})\hookrightarrow\mathcal{R}(L^{1},X_{2})\hookrightarrow \mathcal{R}(X_{2},L^{1}),
\end{align*}
as we wanted to see.
\end{proof}

Now, let us consider the embedding \eqref{eq: R(X1,Loo)->R(X2,Lp1)}, for the case $n\geq 3$.  In particular, we  shall provide a characterization of the smallest mixed norm  space of the form $\mathcal{R}(X_2,L^{1})$ in \eqref{eq: R(X1,Loo)->R(X2,Lp1)} once the mixed norm space $\mathcal{R}(X_{1},L^{\infty})$ is given. In order to do it, we begin with a preliminary lemma. For simplicity, we will assume that $|I|=1$.

\begin{lemma}\label{lemma: optimal range norm R(X,Loo)->R(Y,L1)} Let $X(I^{n-1})$ be an r.i. space, with $n\ge3$. Then, the functional defined by 
 \begin{align}\label{eq: optimal range norm R(X,Loo)->R(Y,L1)}
\bigl\|f\bigr\|_{ X_{\mathcal{R}(X,L^{\infty})}(I^{n-1})}=\bigl\|f^{*}(t^{(n-1)^{\prime}}) \bigr\|_{\overline{X}(0,1)}
,\ \ f\in\mathcal{M}_{+}(I^{n-1}),
 \end{align}
 is an r.i. norm.
 \end{lemma}
 
 \begin{proof}
The positivity and homogeneity of $\bigl\|\cdot\bigr\|_{X_{\mathcal{R}(X,L^{\infty})}(I^{n-1})}$ are clear. Next, let $f$ and $g$ be  measurable functions on $I^{n}.$ Then,
\begin{align*}
\bigl\|f+g\bigr\|_{X_{\mathcal{R}(X,L^{\infty})}(I^{n-1})}&=\bigl\|(f+g)^{*}(t^{(n-1)^{\prime}})\bigr\|_{\overline{X}(0,1)}\\
&=\frac1{(n-1)'}\sup_{\bigl\|h\bigr\|_{X^{\prime}(I^{n-1})}\leq 1}\int^{1}_{0}(f+g)^{*}(t)t^{-1/(n-1)} h^{*}(t^{1/(n-1)'})dt.
\end{align*}
Since  $(f+g)^{*}\prec f^{*}+g^{*}$ (cf. \cite[Theorem II.3.4]{Bennett}),  Hardy-Littlewood-P\'olya Principle  implies that
\begin{align*}
\bigl\|f+g\bigr\|_{X_{\mathcal{R}(X,L^{\infty})}(I^{n-1})}&\le \frac1{(n-1)'}\sup_{\bigl\|h\bigr\|_{X^{\prime}(I^{n-1})}\leq 1}\int^{1}_{0}f^{*}(t)t^{-1/(n-1)} h^{*}(t^{1/(n-1)'})dt\\
&\qquad+\frac1{(n-1)'}\sup_{\bigl\|h\bigr\|_{X^{\prime}(I^{n-1})}\leq 1}\int^{1}_{0}g^{*}(t)t^{-1/(n-1)} h^{*}(t^{1/(n-1)'})dt\\
&=\bigl\|f\bigr\|_{X_{\mathcal{R}(X,L^{\infty})}(I^{n-1})}+\bigl\|g\bigr\|_{X_{\mathcal{R}(X,L^{\infty})}(I^{n-1})}.
\end{align*}

The proof of  (A2)-(A4) and (A6)  for $\bigl\|\cdot\bigr\|_{X_{\mathcal{R}(X,L^{\infty})}(I^{n-1})}$ requires only the corresponding axioms for $\bigl\|\cdot\bigr\|_{\overline{X}(0,1)},$ hence we shall omit them. Finally, to prove property (A5), we fix any $f\in\mathcal{M}(I^{n-1}).$ Then,  
\begin{align*}
\bigl\|f^{*}(t^{(n-1)^{\prime}})\bigr\|_{\overline{X}(0,1)}&\gtrsim \int^{1}_{0}f^{*}(t^{(n-1)^{\prime}})dt\geq \int^{1}_{0} f^{*}(t)dt.
\end{align*}
 \end{proof}
 
\begin{theorem}\label{teo: rango R(X1,Loo)->R(X2,L1)} 
 Let $n\geq 3.$ Let $X(I^{n-1})$ be an r.i. space and let $X_{\mathcal{R}(X,L^{\infty})}(I^{n-1})$ be as in \eqref{eq: optimal range norm R(X,Loo)->R(Y,L1)}.
Then, the embedding
\begin{align}\label{eq: 1rango R(X1,Loo)->R(X2,L1)}
\mathcal{R}(X,L^{\infty})\hookrightarrow \mathcal{R}(X_{\mathcal{R}(X,L^{\infty})},L^{1})
\end{align}
 holds. Moreover, $\mathcal{R}(X_{\mathcal{R}(X,L^{\infty})},L^{1})$ is the smallest  space of the form $\mathcal{R}(Y,L^{1})$ that verifies  \eqref{eq: 1rango R(X1,Loo)->R(X2,L1)}.
 \end{theorem}

\begin{proof}Lemma~\ref{lemma: optimal range norm R(X,Loo)->R(Y,L1)} gives us that $X_{\mathcal{R}(X,L^{\infty})}(I^{n-1})$ is an r.i. space equipped with the norm $ \bigl\|\cdot\bigr\|_{X_{\mathcal{R}(X,L^{\infty})}(I^{n-1})}.$ Now, let us see that the embedding \eqref{eq: 1rango R(X1,Loo)->R(X2,L1)} holds. In fact, if $f\in \mathcal{R}(X,L^{\infty})$ then, combining 
\begin{align*}
  L^{\infty}(I^{n})=\mathcal{R}(L^{\infty},L^{\infty})\hookrightarrow \mathcal{R}(L^{\infty},L^{1}),
\end{align*}
with the embedding \eqref{eq: viktor-robert} and using Theorem~\ref{interpolkf}, we deduce that
 \begin{align*}
K(f,t; \mathcal{R}(L^{(n-1)^{\prime},1},L^{1}),\mathcal{R}(L^{\infty},L^{1}))\lesssim K(f,Ct; \mathcal{R}(L^{1},L^{\infty}),L^{\infty}).
\end{align*}
Hence, using  Lemma~\ref{lemma: lower bound for K-functional} and  Theorem~\ref{theorem: K-functional of R(X,Loo) and Loo}, we get 
\begin{align*}
\int^{t}_{0}\psi^{*}_{j}(f,L^{1})(s^{(n-1)^{\prime}})ds\lesssim \sum^{n}_{k=1}\int^{Ct}_{0}\psi^{*}_{k}(f,L^{\infty})(s)ds,\ \ j\in\left\{1,\ldots,n\right\}.
\end{align*}
So, using the Hardy-Littlewood inequality and the subadditivity  of $\bigl\|\cdot\bigr\|_{\overline{X}(0,1)},$ we get
\begin{align*}
\bigl\|f\bigr\|_{\mathcal{R}_{j}(X_{\mathcal{R}(X,L^{\infty})},L^{1})}&=\bigl\|\psi^{*}_{j}(f,L^{1})\bigr\|_{X_{\mathcal{R}(X,L^{\infty})}(I^{n-1})}\\
&=\bigl\|\psi^{*}_{j}(f,L^{1})(s^{(n-1)^{\prime}})\bigr\|_{\overline{X}(0,1)}\\
&\lesssim \sum^{n}_{k=1} \bigl\|\psi^{*}_{k}(f,L^{\infty})(Ct)\bigr\|_{\overline{X}(0,1)}\lesssim\bigl\|f\bigr\|_{\mathcal{R}(X,L^{\infty})},
\end{align*}
for any $ j\in\left\{1,\ldots,n\right\}.$ Therefore, the embedding
\begin{align*}
 \mathcal{R}(X,L^{\infty})\hookrightarrow \mathcal{R}(X_{\mathcal{R}(X,L^{\infty})},L^{1})
 \end{align*}
 holds, as we wanted to see.

Now, let us prove that $\mathcal{R}(X_{\mathcal{R}(X,L^{\infty})},L^{1})$ is the smallest r.i. space for which \eqref{eq: 1rango R(X1,Loo)->R(X2,L1)} holds. That is, let us see  that if  a mixed norm  space $\mathcal{R}(Y,L^{1})$ satisfies 
\begin{align*}
\mathcal{R}(X,L^{\infty})\hookrightarrow \mathcal{R}(Y,L^{1}),
\end{align*}
then
\begin{align}\label{eq: 1rango optimo R(X,Loo)->R(Y,L1)}
\mathcal{R}(X_{\mathcal{R}(X,L^{\infty})},L^{1})\hookrightarrow \mathcal{R}(Y,L^{1}).
\end{align}
We assume that $I=(-a,b),$ with $a,b\in\mathbb R_{+},$  and $0<r<\min(a,b).$ Given any function  $f\in X_{\mathcal{R}(X,L^{\infty})}(I^{n-1}),$ with $\lambda_{f}(0)\leq \omega_{n-1}r^{(n-1)},$  we define 
\begin{align*}
g(x)=\begin{cases}
f^{*}(\omega_{n-1}|\widehat{x_{n}}|^{(n-1)}),& \textnormal{if $(\widehat{x_{n}},x_{n})\in B_{n-1}(0,r)\times I,$}\\
0,&\textnormal{otherwise.}
\end{cases}
\end{align*} 
We fix  $k\in \left\{1,\ldots,n-1\right\}.$  Then,
\begin{align*}
\psi_{k}(g,L^{\infty})(\widehat{x_{k}})=f^{*}(\omega_{n-1}|\widehat{x_{k,n}}|^{(n-1)}),\ \ \  \textnormal{if $(\widehat{x_{n}},x_{n})\in B_{n-2}(0,r)\times I,$}
\end{align*}
and $\psi_{k}(g,L^{\infty})(\widehat{x_{k}})=0$, otherwise. Thus, for any $ k\in\bigl\{1,\ldots,n-1\bigr\},$
\begin{align}\label{eq: 2rango optimo R(X,Loo)->R(Y,Lp1)}
\bigl\|g\bigr\|_{\mathcal{R}_{k}(X,L^{\infty})}\lesssim\bigl\|f^{*}(t^{(n-1)^{\prime}}) \bigr\|_{\overline{X}(0,|I|^{n-1})}= \bigl\|f\bigr\|_{ X_{\mathcal{R}(X,L^{\infty})}(I^{n-1})}.
\end{align}
On the other hand,
\begin{align*}
\psi_{n}(g,L^{\infty})(\widehat{x_{n}})=f^{*}(\omega_{n-1}|\widehat{x_{n}}|^{n-1}),\ \ \  \textnormal{if $\widehat{x_{n}}\in B_{n-1}(0,r),$}
\end{align*}
and $\psi_{k}(g,L^{\infty})(\widehat{x_{k}})=0.$ So,
\begin{align}\label{eq: 3rango optimo R(X,Loo)->R(Y,Lp1)}
\bigl\|g\bigr\|_{\mathcal{R}_{n}(X,L^{\infty})}=\big\| f\big\|_{X(I^{n-1})}\leq \bigl\|f^{*}(t^{(n-1)^{\prime}}) \bigr\|_{\overline{X}(0,|I|^{n-1})}= \bigl\|f\bigr\|_{ X_{\mathcal{R}(X,L^{\infty})}(I^{n-1})}.
\end{align}
Therefore, using \eqref{eq: 2rango optimo R(X,Loo)->R(Y,Lp1)} and \eqref{eq: 3rango optimo R(X,Loo)->R(Y,Lp1)}, we get
 \begin{align}\label{eq: 4rango optimo R(X,Loo)->R(Y,Lp1)}
 \bigl\|g\bigr\|_{\mathcal{R}(X,L^{\infty})}\lesssim\bigl\|f\bigr\|_{ X_{\mathcal{R}(X,L^{\infty})}(I^{n-1})}.
 \end{align}
 So, using $\mathcal{R}(X,L^{\infty})\hookrightarrow\mathcal{R}(Y,L^{1})$ and \eqref{eq: 4rango optimo R(X,Loo)->R(Y,Lp1)}, we get
  \begin{align}\label{eq: 5rango optimo R(X,Loo)->R(Y,L1)}
 \bigl\|g\bigr\|_{\mathcal{R}(Y,L^{1})}\lesssim  \bigl\|f\bigr\|_{X_{\mathcal{R}(X,L^{\infty})}(I^{n-1})}.
 \end{align}
 But, as before,
 \begin{align*}
 \bigl\|g\bigr\|_{\mathcal{R}(Y,L^{1})}\approx \big\| f\big\|_{Y(I^{n-1})}
 \end{align*}
 and hence, using \eqref{eq: 5rango optimo R(X,Loo)->R(Y,L1)}, we get
 \begin{align*}
\bigl\|f\bigr\|_{Y(I^{n-1})}\lesssim \bigl\|f\bigr\|_{X_{\mathcal{R}(X,L^{\infty})}(I^{n-1})}.
\end{align*}
From this, we obtain that any $f\in X_{R(X,L^{\infty})}(I^{n-1})$, with $\lambda_{f}(0) \leq\omega_{n-1}r^{n-1}$, belongs to $Y (I^{n-1})$. The general case can be proved as in the proof of  Theorem ~\ref{teo: mixta inclusion R(X1,Y1)->R(X2,Y2)}. Thus, we have
$X_{R(X,L^{\infty})}(I^{n-1})\hookrightarrow Y (I^{n-1})$,
and so,  Lemma ~\ref{lema: inclusion norma mixta} implies that the embedding \eqref{eq: 1rango optimo R(X,Loo)->R(Y,L1)}
 holds, as we wanted to see.
\end{proof}

\begin{remark}{\rm
Theorem~\ref{teo: rango R(X1,Loo)->R(X2,L1)} can be extended to the case $p>1$ as follows: if we define $\| f\|_{\widetilde X_p(I^{n-1})}=\| f^*(t^{(p(n-1))'})\|_{\overline{X}(0,1)}$, then with a similar proof one can get
\begin{equation}\label{notopt}
\mathcal{R}(X,L^{\infty})\hookrightarrow \mathcal{R}(\widetilde X_p,L^{p,1}),
\end{equation}
although this embedding  is not optimal in general. For example, with $X=L^{p,1}(I^{n-1})$ then $\widetilde X_p(I^{n-1})=L^{p(p(n-1))',1}(I^{n-1})$, but in this case
$$
\mathcal{R}(L^{p,1},L^{\infty})\hookrightarrow \mathcal{R}(L^{p(n-1)',\infty},L^{p,1}),
$$
and clearly $L^{p(n-1)',\infty}(I^{n-1})\hookrightarrow L^{p(p(n-1))',1}(I^{n-1})$.

However, if $X=L^1(I^{n-1})$, then \eqref{notopt} gives
$$
\mathcal{R}(L^{1},L^{\infty})\hookrightarrow \mathcal{R}(L^{(p(n-1))',1},L^{p,1}),
$$
which is indeed optimal (as proved in \cite{Robert-Viktor}).
}
\end{remark}

%We observe that, in fact, we have seen that $X_{\mathcal{R}(X,L^{\infty})}(I^{n-1})$ is continuously embedded into $Y(I^{n-1}),$ which is a stronger condition than
%\eqref{eq: 1rango optimo R(X,Loo)->R(Y,Lp1)}.
%
%To finish this section, we  shall present an  application of Theorem~\ref{teo: rango R(X1,Loo)->R(X2,Lp1)}. In particular, we shall see that there is no a smallest space  of the form $\mathcal{R}(Y,L^{p,1})$  that would render the embedding due to Algervik and Kolyada \cite{Robert-Viktor}
%\begin{align*}
%\mathcal{R}(L^{1},L^{\infty})\hookrightarrow \mathcal{R}(L^{p(n-1)/(p(n-1)-1),1},L^{p,1}),
%\end{align*}
% true.
%
%\begin{corollary} Let $n=2$ and $1<p<\infty$ or $n\geq 3$ and $1\le  p<\infty.$ Then, the mixed norm space $\mathcal{R}(L^{pp_{1}(n-1)/(p(n-1)-p_{1}),q_{1}},L^{p,1}),$ with $1<p_{1}<\infty$ and $1\leq q_{1}\leq \infty$ or $p_{1}=q_{1}=1,$ is the smallest space of the form $\mathcal{R}(Y,L^{1})$ satisfying
% \begin{align*}
% \mathcal{R}(L^{p_{1},q_{1}},L^{\infty})\hookrightarrow \mathcal{R}(L^{pp_{1}(n-1)/(p(n-1)-p_{1}),q_{1}},L^{p,1}).
% \end{align*}
%\end{corollary}
%
%\begin{proof} It follows from Theorems~\ref{teo: rango R(X1,Loo)->R(X2,Lp1)}. We only need to use $L^{p_{1},q_{1}}(I^{n-1})$ instead of $X_{1}(I^{n-1})$.
%\end{proof}

\section{Fournier embeddings}\label{section: Fournier embeddings}
Our main goal, in this section,  is to study the following embedding
\begin{align}\label{eq: introduccion R(X,Loo)->Z}
\mathcal{R}(X,L^{\infty})\hookrightarrow Z(I^{n}).
\end{align}
In particular, we are interested in the following problems:
\begin{enumerate}[(i)]
	\item Given a mixed norm space $\mathcal{R}(X,L^{\infty}),$ we would like to find the largest r.i. range space $Z(I^{n})$  satisfying  \eqref{eq: introduccion R(X,Loo)->Z}.
	\item Now, let us suppose that the range space is given $Z(I^{n}).$ We would like to find  the largest of the form  $\mathcal{R}(X,L^{\infty})$ for which \eqref{eq: introduccion R(X,Loo)->Z} holds.
	\end{enumerate}
 	
The main motivation to consider these questions come from the embedding due to Fournier \cite{Fournier}, which shows that, if $n\geq 2,$
\begin{align}\label{eq: Fournier R(L,Loo)->Ln',1}
\mathcal{R}(L^{1},L^{\infty})\hookrightarrow L^{n^{\prime},1}(I^{n}).
\end{align}

\subsection{Necessary and sufficient conditions}
 Now, our main purpose is to find  necessary and sufficient conditions
on $X(I^{n-1})$  and $Z(I^{n})$ under which we have the embedding \eqref{eq: introduccion R(X,Loo)->Z}.

\begin{theorem}\label{theorem: mixta inclusion R(X,Loo)->Z}  Let $X(I^{n-1})$ and $Z(I^{n})$ be r.i. spaces. Then, the embedding
\begin{align*}
\mathcal{R}(X,L^{\infty})\hookrightarrow Z(I^{n})
\end{align*}
holds, if and only if, 
\begin{align}\label{eq: 1mixta inclusion R(X,Loo)->Z}
\bigl\|f^{*}(t^{1/n^{\prime}})\bigr\|_{\overline{Z}(0,|I|^{n})}\lesssim \bigl\|f^{*}\bigr\|_{\overline{X}(0,|I|^{n-1})},\ \ \ f\in X(I^{n-1}).
\end{align}
\end{theorem}

\begin{proof} Let us first suppose that the embedding
\begin{align*}
\mathcal{R}(X,L^{\infty})\hookrightarrow Z(I^{n})
\end{align*}
holds. As before,  we assume that $I=(-a,b),$ with $a,b\in\mathbb R_{+}$ and $0<r<\min(a,b).$ Given any  $f\in X(I^{n-1}),$ with $\lambda_{f}(0)\leq \omega^{1/n^{\prime}}_{n}r^{n-1},$ we define 
\begin{align*}
g(x)=\begin{cases}
f^{*}(\omega^{1/n^{\prime}}_{n}|x|^{n-1}),& \textnormal{if $x\in B_{n}(0,r),$}\\
0,&\textnormal{otherwise.}
\end{cases}
\end{align*} 
Now, we fix any $k\in \left\{1,\ldots,n\right\}.$ Then,
\begin{align*}
\psi_{k}(g,L^{\infty})(\widehat{x_{k}})=f^{*}(\omega^{1/n^{\prime}}_{n}|\widehat{x_{k}}|^{n-1}),\ \ \textnormal{for any $\widehat{x_{k}}\in B_{n-1}(0,r),$}
\end{align*}
and $\psi_{k}(g,L^{\infty})(\widehat{x_{k}})=0$ otherwise. Thus, using the boundedness of the dilation operator in r.i. spaces, we get
\begin{align*}
\bigl\|g\bigr\|_{\mathcal{R}_{k}(X,L^{\infty})}=\bigl\|\psi_{k}(g,L^{\infty})\bigr\|_{X(I^{n-1})}\lesssim \bigl\|f^{*}\bigr\|_{\overline{X}(0,|I|^{n-1})},\ \ k\in\left\{1,\ldots,n\right\},
\end{align*}
So, our assumption on $f$ shows that $g\in \mathcal{R}(X,L^{\infty})$ and  
\begin{align}\label{eq: 2mixta inclusion R(X,Loo)->Z}
\bigl\|g\bigr\|_{\mathcal{R}(X,L^{\infty})}&\lesssim \bigl\|f^{*}\bigr\|_{\overline{X}(0,|I|^{n-1})}.
\end{align}
 Thus, using $
\mathcal{R}(X,L^{\infty})\hookrightarrow Z(I^{n})$ and \eqref{eq: 2mixta inclusion R(X,Loo)->Z}, we obtain
\begin{align}\label{eq: 3mixta inclusion R(X,Loo)->Z}
\bigl\|g\bigr\|_{Z(I^{n})}\lesssim \bigl\|f^{*}\bigr\|_{\overline{X}(0,|I|^{n-1})}.
\end{align}
 But,   
 \begin{align*}
 \bigl\|g\bigr\|_{Z(I^{n})}=\bigl\|f^{*}(t^{1/n^{\prime}})\bigr\|_{\overline{Z}(0,|I|^{n})},
 \end{align*}
  and hence \eqref{eq: 3mixta inclusion R(X,Loo)->Z} gives
\begin{align*}
\bigl\|f^{*}(t^{1/n^{\prime}})\bigr\|_{\overline{Z}(0,|I|^{n})}\lesssim  \bigl\|f\bigr\|_{\overline{X}(0,|I|^{n-1})}.
\end{align*}
 This proves \eqref{eq: 1mixta inclusion R(X,Loo)->Z}, for any function $f\in X(I^{n-1}),$ with $\lambda_{f}(0)\leq \omega^{1/n^{\prime}}_{n}r^{n-1}.$ The general case can be proved as in the proof of Theorem~\ref{teo: mixta inclusion R(X1,Y1)->R(X2,Y2)}.

Now, let us suppose that  \eqref{eq: 1mixta inclusion R(X,Loo)->Z} holds.  We fix any $f\in \mathcal{R}(X,L^{\infty})$ and $s\in (0,|I|^{n}).$ Then, for any $k\in \left\{1,\ldots,n\right\},$ it holds that
\begin{align*}
\lambda_{\psi_{k}(f,L^{\infty})}\Bigl(\sum^{n}_{j=1}\psi^{*}_{j}(f,L^{\infty})(s^{1/n^{\prime}})\Bigr)\leq\lambda_{\psi_{k}(f, L^{\infty})}(\psi^{*}_{k}(f,L^{\infty})(s^{1/n^{\prime}}))\leq s^{1/n^{\prime}}.
\end{align*}
So, we get
\begin{align*}
\prod^{n}_{k=1}\lambda_{\psi_{k}(f,L^{\infty})}\Bigl(\sum^{n}_{j=1}\psi^{*}_{j}(f,L^{\infty})(s^{1/n^{\prime}})\Bigr)\leq s^{n-1}.
\end{align*}
Hence, Corollary~\ref{corollary: distribution and essential projection }, with $t$ replaced by $\sum^{n}_{j=1}\psi^{*}_{j}(f,L^{\infty})(s^{1/n^{\prime}}),$ implies that
\begin{align*}
\lambda_{f}\Bigl(\sum^{n}_{j=1}\psi^{*}_{j}(f,L^{\infty})(s^{1/n^{\prime}})\Bigr)\leq \Bigl(\prod^{n}_{k=1}\lambda_{\psi_{k}(f,L^{\infty})}\Bigl(\sum^{n}_{j=1}\psi^{*}_{j}(f,L^{\infty})(s^{1/n^{\prime}})\Bigr)\Bigr)^{1/(n-1)}\leq s,
\end{align*}
for any $0<s<|I|^{n}.$ As a consequence, 
\begin{align*}
\sum^{n}_{j=1}\psi^{*}_{j}(f,L^{\infty})(s^{1/n^{\prime}})\in \bigl\{y\geq0: \lambda_{f}(y)\leq s\bigr\},\ \  \textnormal{for any $0<s<|I|^{n}.$}
\end{align*}
Therefore, we obtain
\begin{align}\label{refrem}
f^{*}(s)\leq \sum^{n}_{j=1}\psi^{*}_{j}(f,L^{\infty})(s^{1/n^{\prime}}),\ \  \textnormal{for any $0<s<|I|^{n}.$}
\end{align}
 Thus, we have
\begin{align*}
\bigl\|f\bigr\|_{Z(I^{n})}\leq \biggl\|\sum^{n}_{k=1}\psi^{*}_{k}(f, L^{\infty})(s^{1/n^{\prime}})\biggr\|_{\overline{Z}(0,|I|^{n})}\leq \sum^{n}_{k=1}\bigl\|\psi^{*}_{k}(f, L^{\infty})(s^{1/n^{\prime}})\bigr\|_{\overline{Z}(0,|I|^{n})}.
\end{align*}
Hence, using  \eqref{eq: 1mixta inclusion R(X,Loo)->Z}, we get
\begin{align*}
\bigl\|f\bigr\|_{Z(I^{n})}\leq \sum^{n}_{k=1}\bigl\|\psi^{*}_{k}(f, L^{\infty})(s^{1/n^{\prime}})\bigr\|_{\overline{Z}(0,|I|^{n})}\lesssim\sum^{n}_{k=1}\bigl\|\psi_{k}(f, L^{\infty})\bigr\|_{X(I^{n-1})}=\bigl\|f\bigr\|_{\mathcal{R}(X,L^{\infty})}.
\end{align*}
 That is, the embedding $\mathcal{R}(X,L^{\infty})\hookrightarrow Z(I^{n})$ holds and the proof is complete.
\end{proof}

\subsection{The optimal domain problem}

Let $Z(I^{n})$ be an r.i. space. Now, we want to find the largest space of the form $\mathcal{R}(X,L^{\infty})$ satisfying
$$R(X,L^{\infty})\hookrightarrow Z(I^{n}).$$
In order to do this, let us introduce a new space, denoted by $X_{Z,L^{\infty}}(I^{n-1})$, consisting of those functions $f\in\mathcal{M}(I^{n-1})$  for which the quantity
\begin{align}\label{eq: optimal domain norm}
\bigl\|f\bigr\|_{X_{Z,L^{\infty}}(I^{n-1})}= \bigl\|f^{**}(s^{1/n^{\prime}})\bigr\|_{\overline{Z}(0,|I|^{n})}
\end{align}
is finite. It is not difficult to verify that  $X_{Z,L^{\infty}}(I^{n-1})$ is an r.i. space equipped with the norm $\bigl\|\cdot\bigr\|_{X_{Z,L^{\infty}}(I^{n-1})}.$ 

The next lemma  gives an equivalent expression for the norm $\bigl\|\cdot\bigr\|_{X_{Z,L^{\infty}}(I^{n-1})}.$  The  proof follows  the same ideas of \cite[Theorem 4.4]{Edmunds-Kerman-Pick}, so we do not include it here. 

 \begin{lemma}\label{lemma: funcional equivalente dominio R(X,Loo)->Z} Let $Z(I^{n})$ be an r.i space, with  $\overline{\alpha}_{Z}<1/n^{\prime}.$
 Then, 
 \begin{align*}
\bigl\|f\bigr\|_{X_{Z,L^{\infty}}(I^{n-1})}\approx \bigl\|f^{*}(t^{1/n^{\prime}})\bigr\|_{\overline{Z}(0,|I|^{n})},\ \ \textnormal{for any $f\in \mathcal{M}(I^{n-1}).$}
 \end{align*}
 \end{lemma}
 
 \begin{theorem}\label{theorem: dominio R(X,Loo)->Z} Let $Z(I^{n})$ be an r.i. space, with $\overline\alpha_{Z}<1/n^{\prime},$ and let $X_{Z, L^{\infty}}(I^{n-1})$ be the r.i. space defined in \eqref{eq: optimal domain norm}
 Then, the embedding 
 \begin{align}\label{eq: 1dominio R(X,Loo)->Z}
\mathcal{R}(X_{Z,L^{\infty}},L^{\infty})\hookrightarrow Z(I^{n})
\end{align}
holds. Moreover, $\mathcal{R}(X_{Z,L^{\infty}},L^{\infty})$ is the largest space of the form $\mathcal{R}(X,L^{\infty})$ for which the embedding \eqref{eq: 1dominio R(X,Loo)->Z} holds.
 \end{theorem}
 
 \begin{proof}  The embedding \eqref{eq: 1dominio R(X,Loo)->Z}  follows from  Theorem~\ref{theorem: mixta inclusion R(X,Loo)->Z}. Thus, to complete the proof, it only remains to see that  $\mathcal{R}(X_{Z,L^{\infty}},L^{\infty})$ is the largest domain space of the form $\mathcal{R}(X,L^{\infty})$ corresponding to $Z(I^{n})$. In fact, we shall see that that if  $\mathcal{R}(Y,L^{\infty})$ is another mixed norm space  such that \eqref{eq: 1dominio R(X,Loo)->Z} holds with $\mathcal{R}(X_{Z,L^{\infty}},L^{\infty})$ replaced by $\mathcal{R}(Y,L^{\infty})$, then
  \begin{align*}
\mathcal{R}(Y,L^{\infty})\hookrightarrow \mathcal{R}(X_{Z,L^{\infty}},L^{\infty}).
\end{align*}
 We fix any $\mathcal{R}(Y,L^{\infty}).$ Then,  Theorem~\ref{theorem: mixta inclusion R(X,Loo)->Z} ensures us that 
 \begin{align*}
\bigl\|f^{*}(t^{1/n^{\prime}})\bigr\|_{\overline{Z}(0,|I|^{n})}\lesssim \bigl\|f\bigr\|_{X(I^{n-1})},
\end{align*}
and so, using Lemma~\ref{lemma: funcional equivalente dominio R(X,Loo)->Z}, we get   
\begin{align*}
\bigl\|f\bigr\|_{X_{Z,L^{\infty}}(I^{n-1})}\lesssim \bigl\|f\bigr\|_{X(I^{n-1})},  \ \ \ \ f\in X(I^{n-1}).
\end{align*}
 That is, $X(I^{n-1})\hookrightarrow X_{Z,L^{\infty}}(I^{n-1}).$ Hence, using Theorem~\ref{teo: mixta inclusion R(X1,Y1)->R(X2,Y2)}, we deduce that 
  \begin{align*}
 \mathcal{R}(X,L^{\infty})\hookrightarrow \mathcal{R}(X_{Z,L^{\infty}},L^{\infty}).
 \end{align*}
  as we wanted to see.
\end{proof}

Let us see  an application of Theorem~\ref{theorem: dominio R(X,Loo)->Z}.

\begin{corollary}Let $n^{\prime}<p_{1}<\infty,$ and $1\leq q_{1}\leq \infty.$ Then, the mixed norm space $\mathcal{R}(L^{p_{1}/n^{\prime},q_{1}},L^{\infty})$ is the largest space of the form $\mathcal{R}(X,L^{\infty})$ satisfying
\begin{align*}
\mathcal{R}(L^{p_{1}/n^{\prime},q_{1}},L^{\infty})\hookrightarrow L^{p_{1},q_{1}}(I^{n}).
\end{align*}
\end{corollary}
\begin{proof} It follows from Theorem~\ref{theorem: dominio R(X,Loo)->Z}, with $Z(I^{n})$ replaced by $L^{p_{1},q_{1}}(I^{n}).$
\end{proof}

\subsection{The optimal range problem}
Let $X(I^{n-1})$ be an r.i. space. We would like to describe the smallest r.i. space $Z(I^{n})$ satisfying  
\begin{align*}
\mathcal{R}(X,L^{\infty})\hookrightarrow Z(I^{n}).
\end{align*}
We begin with a preliminary lemma.

\begin{lemma}\label{lemma: optimal range norm} Let $X(I^{n-1})$ be an r.i. space. Then, the functional defined by 
 \begin{align}\label{eq: optimal range norm}
\bigl\|f\bigr\|_{Z_{\mathcal{R}(X,L^{\infty})}(I^{n})}= \bigl\|f^{*}(t^{n^{\prime}})\bigr\|_{\overline{X}(0,|I|^{n-1})},\ \ f\in\mathcal{M}_{+}(I^{n}),
 \end{align}
 is an r.i. norm.
 \end{lemma}
 
 \begin{proof}It is enough to apply the same technique as in the proof of Lemma~\ref{lemma: optimal range norm R(X,Loo)->R(Y,L1)} .
 \end{proof}
 
 \begin{theorem}\label{teo: rango R(X,Loo)->Z} Let $X(I^{n-1})$ be an r.i. space and let $Z_{\mathcal{R}(X,L^{\infty})}(I^{n})$ be as in \eqref{eq: optimal range norm}
Then, the embedding
\begin{align*}
\mathcal{R}(X,L^{\infty})\hookrightarrow Z_{\mathcal{R}(X,L^{\infty})}(I^{n})
\end{align*}
 holds. Moreover, $Z_{\mathcal{R}(X,L^{\infty})}(I^{n})$ is the smallest r.i. space that verifies  this embedding.
 \end{theorem}
 
 \begin{proof} Lemma~\ref{lemma: optimal range norm} gives us that $Z_{\mathcal{R}(X,L^{\infty})}(I^{n})$ is an r.i. space equipped with the norm $ \bigl\|\cdot\bigr\|_{Z_{\mathcal{R}(X,L^{\infty})}(I^{n})}.$ Now, let us see that the embedding
 \begin{align*}
\mathcal{R}(X,L^{\infty})\hookrightarrow Z_{\mathcal{R}(X,L^{\infty})}(I^{n})
\end{align*}
holds. In fact,  let $f$ be any function from $\mathcal{R}(X,L^{\infty}).$ Then, as a consequence of  \eqref{refrem}, we obtain

\begin{align*}
\int^{t}_{0}f^{*}(s^{n^{\prime}})ds\lesssim \sum^{n}_{k=1}\int^{t}_{0}\psi^{*}_{k}(f,L^{\infty})(s)ds,\ \  0<t<|I|^{n-1}.
\end{align*}
Therefore, using Hardy-Littlewood-P\'olya Principle and the subadditive property of $\bigl\|\cdot\bigr\|_{\overline{X}(0,|I|^{n-1})},$ we get
\begin{align*}
\bigl\|f^{*}(s^{n^{\prime}})\bigr\|_{\overline{X}(0,|I|^{n-1})}\lesssim \sum^{n}_{k=1} \bigl\|\psi^{*}_{k}(f,L^{\infty})\bigr\|_{\overline{X}(0,|I|^{n-1})}=\bigl\|f\bigr\|_{\mathcal{R}(X,L^{\infty})}.
\end{align*}
 That is, $\mathcal{R}(X,L^{\infty})\hookrightarrow Z_{\mathcal{R}(X,L^{\infty})}(I^{n}).$

Now, let us see that $Z_{\mathcal{R}(X,L^{\infty})}(I^{n})$ is the smallest r.i. space for which this embedding holds, i.e., let us see  that if  an r.i. space $Z(I^{n})$ satisfies 
\begin{align*}
\mathcal{R}(X,L^{\infty})\hookrightarrow Z(I^{n}),
\end{align*}
then $Z_{\mathcal{R}(X,L^{\infty})}(I^{n})\hookrightarrow Z(I^{n}).$ As before, assume that $I=(-a,b),$ with $a,b\in\mathbb R_{+}$ and  $0<r<\min(a,b).$ Given any function $f\in Z_{\mathcal{R}(X,L^{\infty})}(I^{n}),$ with $\lambda_{f}(0)\leq \omega_{n}r^{n},$ we define  
\begin{align*}
g(x)=\begin{cases}
f^{*}(\omega_{n}|x|^{n}),& \textnormal{if $x\in B_{n}(0,r),$}\\
0,&\textnormal{otherwise.}
\end{cases}
\end{align*} 
Then, applying the same technique as in the proof of  Theorem~\ref{teo: mixta inclusion R(X1,Y1)->R(X2,Y2)} and using the boundedness of the dilation operator in r.i. spaces, we get
\begin{align}\label{eq: 1rango R(X,Loo)->Z} 
\bigl\|g\bigr\|_{\mathcal{R}(X,L^{\infty})}\lesssim \bigl\|f^{*}(t^{n^{\prime}})\bigr\|_{\overline{X}(0,|I|^{n-1})}=\bigl\|f\bigr\|_{Z_{\mathcal{R}(X,L^{\infty})}(I^{n})}.
\end{align}
By hypothesis   $f\in Z_{\mathcal{R}(X,L^{\infty})}(I^{n}),$  and hence $g\in \mathcal{R}(X,L^{\infty}).$ So, using the embedding $\mathcal{R}(X,L^{\infty})\hookrightarrow Z(I^{n})$ and  \eqref{eq: 1rango R(X,Loo)->Z} we get
\begin{align*}
\bigl\|g\bigr\|_{Z(I^{n})}\lesssim\bigl\|f\bigr\|_{Z_{\mathcal{R}(X,L^{\infty})}(I^{n})}.
\end{align*}
But,   $g$ and $f$ are equimeasurable functions, and  hence  we obtain
\begin{align*}
\bigl\|f\bigr\|_{Z(I^{n})}\lesssim \bigl\|f\bigr\|_{Z_{\mathcal{R}(X,L^{\infty})}(I^{n})}.
\end{align*}
 From this,  we obtain that any  $f\in Z_{\mathcal{R}(X,L^{\infty})}(I^{n}), $ with $\lambda_{f}(0)\leq \omega_{n}r^{n},$ belongs to  $Z(I^{n}).$ The general case can be proved as in the proof of  Theorem~\ref{teo: mixta inclusion R(X1,Y1)->R(X2,Y2)}. Thus, the proof is complete.
 \end{proof}
 
 We shall give now a corollary of Theorems~\ref{teo: rango R(X,Loo)->Z}. In particular, we shall see that  the Fournier's embedding \eqref{eq: Fournier R(L,Loo)->Ln',1} cannot be improved within the class of r.i. spaces. This should be understood as follows: if we replace the range space  in
 \begin{align*}
\mathcal{R}(L^{1},L^{\infty})\hookrightarrow L^{n^{\prime},1}(I^{n}),
\end{align*}
 by a smaller r.i. space, say $Y(I^{n})$, then the resulting embedding
\begin{align*}
\mathcal{R}(L^{1},L^{\infty})\hookrightarrow Y(I^{n}).
\end{align*}
can no longer be true.

 \begin{corollary}Let $1<p_{1}<\infty$ and $1\leq q_{1}\leq \infty$ or  $p_{1}=q_{1}=1.$ Then, the Lorentz space $L^{n^{\prime}p_{1},q_{1}}(I^{n})$ is the smallest r.i. space satisfying
 \begin{align*}
 \mathcal{R}(L^{p_{1},q_{1}},L^{\infty})\hookrightarrow L^{n^{\prime}p_{1},q_{1}}(I^{n}).
 \end{align*}
  \end{corollary}
 
 \begin{proof} 
 It follows from Theorems~\ref{teo: rango R(X,Loo)->Z}. We only need to use $L^{p_{1},q_{1}}(I^{n-1})$ instead of $X(I^{n-1})$.
 \end{proof}

\noindent
{\bf Acknowledgments:} We would like to thank the referee for his/her careful revision which has improved the final version of this work.

\end{document}